\documentclass[11pt]{article}

\usepackage{fullpage}
\usepackage[T1]{fontenc} 
\usepackage{amsmath,amssymb}
\usepackage{amsopn,amsthm}

\usepackage{graphicx} 
\usepackage{color}
\usepackage{psfrag}

\usepackage{tikz}

\usepackage{subcaption}
\usepackage{nicefrac}

\usepackage{svg}
\usepackage{tikz}
\usetikzlibrary{positioning,arrows.meta}
\usepackage{mathrsfs} 
\usepackage{bm}
\usepackage{cases}

\usepackage[colorinlistoftodos]{todonotes}

\usepackage{algpseudocode} 
\usepackage{algorithm}
\floatstyle{plain} \newfloat{myalgo}{tbhp}{mya}

{\begin{myalgo}[#1]
    \centering
    \begin{minipage}{0.9\textwidth}
      \begin{algorithm}[H]}%
      {\end{algorithm}
    \end{minipage}
  \end{myalgo}}

\usepackage[colorlinks=true]{hyperref}
\hypersetup{urlcolor=blue, citecolor=blue, linkcolor=blue}

\usepackage{enumitem}
\newtheorem{theorem}{Theorem}[section]
\newtheorem{proposition}[theorem]{Proposition}
\newtheorem{lemma}[theorem]{Lemma}
\newtheorem{corollary}[theorem]{Corollary}

\theoremstyle{definition}

\newtheorem{example}[theorem]{Example}

\theoremstyle{remark} \newtheorem{remark}[theorem]{Remark}

\numberwithin{equation}{section}
\numberwithin{figure}{section}
\numberwithin{algorithm}{section}

\newcommand{\field}[1]{{\mathbb{#1}}}
\newcommand{\C}{\field{C}}
\newcommand{\N}{\field{N}}
\newcommand{\R}{\field{R}}
\newcommand{\Acal}{\mathcal{A}}
\newcommand{\Bcal}{\mathcal{B}}
\newcommand{\Ccal}{\mathcal{C}}
\newcommand{\Dcal}{\mathcal{D}}

\newcommand{\Fcal}{\mathcal{F}}
\newcommand{\Gcal}{\mathcal{G}}

\newcommand{\Lcal}{\mathcal{L}}

\newcommand{\Pcal}{\mathcal{P}}

\newcommand{\Xcal}{\mathcal{X}}

\newcommand{\Lscr}{\mathscr{L}}
\newcommand{\Dscr}{\mathscr{D}}
\newcommand{\Sscr}{\mathscr{S}}

\newcommand{\loc}{{\mathrm{loc}}}
\newcommand{\ol}[1]{\overline{#1}}

\newcommand{\subs}{\subseteq} 
\newcommand{\trans}{{\top}}

\newcommand{\ds}{\, \dif A} 
\newcommand{\dt}{\, \dif t}
 
\newcommand{\dx}{\, \dif V}

\newcommand{\dtau}{\, \dif \tau}

\newcommand{\ghat}{\widehat g}

\newcommand{\uhat}{\widehat u}
\newcommand{\vhat}{\widehat v} 
\newcommand{\what}{\widehat w}

\newcommand{\zhat}{\widehat{z}}

\newcommand{\Fhat}{\widehat{F}}
\newcommand{\What}{\widehat{W}}

\newcommand{\etahat}{\widehat\eta}
\newcommand{\phihat}{\widehat\phi}

\newcommand{\zetahat}{\widehat\zeta}

\newcommand{\uihat}{\uhat^i}
\newcommand{\ushat}{\uhat^s}
\newcommand{\uthat}{\uhat}
\newcommand{\uinftyhat}{\uhat^\infty}

\newcommand{\vinftyhat}{\vhat^\infty}

\newcommand{\wtilde}{{\widetilde w}}

\newcommand{\Rtilde}{{\widetilde R}}

\newcommand{\rmi}{\mathrm{i}} 

\newcommand{\rme}{\mathrm{e}}

\newcommand{\eps}{\varepsilon}

\newcommand{\Sd}{{S^2}}

\newcommand{\Rd}{{\R^3}}

\renewcommand{\div}{\mathrm{div}}

\DeclareMathOperator{\sgrad}{Grad}

\DeclareMathOperator{\sdiv}{Div}

\newcommand{\BR}{{B_R(0)}}
\newcommand{\BRtilde}{{B_{\Rtilde}(0)}}

\newcommand{\finc}{f^i}
\newcommand{\ui}{u^i}
\newcommand{\us}{u^s}
\newcommand{\ut}{u}
\newcommand{\uinfty}{u^\infty}

\newcommand{\ph}{\,\cdot\,}

\newcommand{\Dbar}{\overline{D}}
\newcommand{\Dplus}{\Rd\setminus\Dbar}

\newcommand{\di}{\partial}
\newcommand{\dit}{\partial_t}

\newcommand{\SL}{\mathrm{SL}}
\newcommand{\DL}{\mathrm{DL}}
\newcommand{\SLop}{S}
\newcommand{\DLop}{K}
\newcommand{\adjDLop}{K'}
\newcommand{\hypSingop}{T}

\newcommand{\Hminhalf}{H^{-\frac12}}
\newcommand{\Hhalf}{H^{\frac12}}
\newcommand{\Hthreehalf}{H^{\frac32}}
\newcommand{\HminhalfdiD}{\Hminhalf(\di D)}
\newcommand{\HhalfdiD}{\Hhalf(\di D)}
\newcommand{\HthreehalfdiD}{\Hthreehalf(\di D)}

\newcommand{\snorm}{{\vert\kern-0.25ex\vert\kern-0.25ex\vert}}
\newcommand{\bigsnorm}{{\big\vert\kern-0.25ex\big\vert\kern-0.25ex\big\vert}}
\newcommand{\Bigsnorm}{{\Big\vert\kern-0.25ex\Big\vert\kern-0.25ex\Big\vert}}

\DeclareMathAlphabet{\mathbi}{\encodingdefault}{\rmdefault}{\bfdefault}{\itdefault}
\DeclareRobustCommand{\vec}[1]{\ifmmode\mathbi{#1}\else\textbf{\textit{#1}}\fi}

\DeclareMathOperator{\dif}{d\!}  
\DeclareMathOperator{\dist}{dist}

\DeclareMathOperator{\supp}{supp}

\DeclareMathOperator{\real}{Re}

\begin{document}

\title{Domain derivative and shape reconstruction for an
  inverse backscattering problem for the wave equation} 
\author{Roland Griesmaier\footnote{Institut f\"ur Angewandte und Numerische
    Mathematik, Karlsruher Institut f\"ur Technologie, Karlsruhe,
    Germany ({\tt roland.griesmaier@kit.edu, eliane.kummer@kit.edu})}\,\! ,
  Marvin Kn\"oller\footnote{Department of Mathematics and Statistics,
    University of Helsinki, Helsinki, Finland
    ({\tt marvin.knoller@helsinki.fi})}\,\! , and
  Eliane Kummer\footnotemark[1]\,\! .
}
\date{\today}

\maketitle

\begin{abstract}
  We consider an inverse backscattering problem for time-dependent
  acoustic waves with a compactly supported penetrable scattering
  obstacle in unbounded three-dimensional free space.
  Assuming that dynamic backscattering far field data of scattered
  waves corresponding to a few time-dependent incident plane waves are 
  available, the goal is to reconstruct the shape of the scattering
  obstacle. 
  We establish Fr\'echet differentiability of the time-dependent far
  field pattern with respect to the shape of the scattering obstacle
  and give a characterization of the associated temporal domain
  derivative in terms of its Laplace transform. 
  This characterization is then utilized in an efficient
  implementation of a regularized Gau{\ss}-Newton method for 
  the inverse backscattering problem using convolution quadrature and 
  a boundary element method. 
  Numerical examples demonstrate potentials and limitations of the
  algorithm. 
\end{abstract}

% Text of article.
{\small\noindent
  Mathematics subject classifications (MSC2020): 35R30, (65N21)
  \\\noindent
  Short title: Time-dependent inverse backscattering
  % \\[1em]\noindent Last modified: \today
}

%%%%%%%%%%%%%%%%%%%%%%%%%%%%%%%%%%%%%%%%%%%%%%%%%%%%%%%%%%%%%%%%%%%%%% 
\section{Introduction}
\label{sec:Introduction}
%%%%%%%%%%%%%%%%%%%%%%%%%%%%%%%%%%%%%%%%%%%%%%%%%%%%%%%%%%%%%%%%%%%%%% 
While numerical reconstruction algorithms for inverse obstacle
scattering problems for time-harmonic acoustic waves have already been
thoroughly investigated (see, e.g., \cite{ColKre19}), corresponding 
algorithms for time-depen\-dent scattering data received somewhat
less attention so far. 
However, in practice it is actually very natural to consider
time-dependent incident waves and to use remote observations of the
associated scattered waves over some finite time interval for
reconstruction. 
In particular dynamic backscattering data can be measured efficiently
using just one source and one receiver at the same observation point. 
Moving this source-receiver combination to different observation 
points around the obstacle, such backscattering data sets can be 
obtained for different incident/observation directions. 
The goal of this work is to establish an efficient numerical
reconstruction scheme for such a dynamic backscattering setup.

We will consider an idealized setting, where a causal incident plane
wave interacts with a compactly supported penetrable scattering
obstacle that is characterized by its constant wave speed and constant
density, both of which may differ from the corresponding parameters of 
the surrounding unbounded three-dimensional free space. 
To model remote observations of the scattered wave in the
backscattering direction, we assume that the time-dependent far field
pattern of the scattered wave is given in the direction opposite to the
direction of propagation of the corresponding incident plane wave for
some finite time interval.
Moreover, we suppose that such a one-dimensional backscattering data 
set is available for a few incident plane waves with different 
incident directions around the obstacle. 

Since the reconstruction method proposed in this work builds heavily on 
ideas from the time-harmonic setting, we briefly review the corresponding 
inverse obstacle scattering problem in the time-harmonic regime to place 
our results in context.
Basically two main classes of reconstruction methods for inverse
obstacle scattering problems have emerged over the years. 
The first class are iterative regularization methods that reformulate
the inverse obstacle scattering problem as a nonlinear shape
optimization problem and then iteratively update the domain of the
unknown scatterer to arrive at a solution.
These updates are usually determined using the domain derivative
associated to the direct scattering problem (see,
e.g.,~\cite{CosLeL12a,CosLeL12b,FinHet24,Het95,Het12,HipLi18,Kir93,Pot94}), 
which requires solving the direct scattering problem at each step 
of the iteration.
We refer, e.g., 
to~\cite{EHHS19,FarTeyDje02,HABH19,HagHet20,HarHog07,IvaLeL16,SRHW25} 
for some more recent contributions on iterative reconstruction schemes 
for time-harmonic waves. 
An iterative scheme for the inverse backscattering problem has 
been considered in~\cite{KreRun98}. 
The second class of methods are qualitative reconstruction schemes,
which are non-iterative and avoid solving direct scattering problems. 
Those usually require less a priori information on the unknown
scattering obstacle but on the other hand use much larger data sets
than iterative reconstruction schemes (see, e.g.,
\cite{CakCol14,CakColHad23,KirGri08}).
Since they typically rely on some sort of sampling strategy for
reconstruction, no parametrization of the unknown scatterer is
required.

In recent years several qualitative reconstruction methods
have been transferred from the frequency domain to the time domain.
For instance linear sampling methods have been proposed
in~\cite{CakMonSel21,CHLM10,HadLecMam14}, factorization methods have 
been considered in~\cite{CakHadLec19,HadLiu20}, and the enclosure
method has been extended in various directions
in~\cite{Ike10,Ike12}.
In this work we instead consider iterative regularization and develop
an efficient reconstruction method that is based on temporal domain derivatives.
We generalize a corresponding analysis and an algorithm
from~\cite{KnoNic26}, where a two-dimensional inverse obstacle
scattering problem with a sound-soft obstacle and time-dependent near
field scattering data corresponding to a single incident plane wave
has been considered.
The main novelties compared to~\cite{KnoNic26} are that we consider a
three-dimensional instead of a two-dimensional scattering problem, 
far field backscattering data instead of near field data 
corresponding to a single incident wave, and a penetrable instead of an impenetrable scattering obstacle. 
We develop a rigorous characterization of the temporal domain
derivative for the associated direct scattering problem in terms of
its Laplace transform. 
To this end, we integrate the variational analysis of transient scattering problems via the Laplace transform, as introduced in \cite{BamHaD86}, with a generalization of the Fr\'echet differentiability and shape derivative results of \cite{Het95,Kir93} to complex frequencies.
The time-domain derivative is subsequently incorporated into an iterative Gauss--Newton-type reconstruction algorithm, enabling an efficient implementation based on convolution quadrature and boundary element methods~\cite{BanLubMel11,Lub94}.
We note that \cite{DonSuZha26, ZhaDonChe26, ZhaDonMa21,ZhaDonMa22} also consider an iterative reconstruction method for a time-dependent inverse scattering problem. In contrast to our approach, these works first apply a convolution quadrature discretization in time and then perform the shape reconstruction sequentially at discrete Laplace domain frequencies, thereby avoiding an analysis in the time domain.

The outline of this paper is as follows.
In Section~\ref{sec:Setting} we introduce the transient acoustic scattering
problem that we are going to consider throughout this article and we
introduce the corresponding far field expansion of the time-dependent scattered wave.
We analyze the direct scattering problem first in the Laplace domain
in Section~\ref{sec:ScatteringProblemLD} and derive bounds on the domain
derivative that are explicit in the Laplace parameter in Section~\ref{sec:DomainDerivativeLD}. 
These results are transferred from the Laplace domain to the time
domain using a Paley-Wiener argument in Section~\ref{sec:TimeDomain},
and Section~\ref{sec:NumericalExamples} is devoted to the numerical
solution of the inverse backscattering problem.
We present numerical results. 
In the appendix we collect some technical estimates that are used
throughout the article.

%%%%%%%%%%%%%%%%%%%%%%%%%%%%%%%%%%%%%%%%%%%%%%%%%%%%%%%%%%%%%%%%%%%%%% 
\section{Problem setting}
\label{sec:Setting}
%%%%%%%%%%%%%%%%%%%%%%%%%%%%%%%%%%%%%%%%%%%%%%%%%%%%%%%%%%%%%%%%%%%%%% 
We consider scattering of a time-dependent acoustic incident plane
wave
\begin{equation}
  \label{eq:PlaneWave}
  \ui(x,t;\theta)
  \,:=\, \finc(t-\theta\cdot x) \,, \qquad (x,t)\in \Rd\times \R \,,
\end{equation}
with direction of propagation $\theta\in\Sd$ and for some compactly
supported $\finc\in H^r_c(\R)$, $r\geq 2$, by a penetrable obstacle
supported in a bounded $C^2$-domain $D\subs\Rd$.
We suppose that the wave speed~$c$ and the density $\rho$ of the
medium are given by the piecewise constant functions 
\begin{equation*}
  c(x) \,:=\,
  \begin{cases}
    c_0 \,, &x\in D \,,\\
    1 \,, & x\in\Dplus \,,
  \end{cases}
  \qquad\text{ and }\qquad
  \rho(x) \,:=\,
  \begin{cases}
    \rho_0 \,, &x\in D \,,\\
    1 \,, & x\in\Dplus \,,
  \end{cases}
\end{equation*}
where $c_0,\rho_0 > 0$ are the material parameters inside the
scatterer. 
Throughout we denote by $\BR$ a sufficiently large ball around the
origin containing $D$ in its interior. 
The incident wave satisfies
\begin{subequations}
  \label{eq:WaveEquationUi}
  \begin{equation}
    \dit^2\ui - \Delta\ui
    \,=\, 0 \qquad \text{in } \Rd\times\R \,,
  \end{equation}
  and assuming that $\ui$ satisfies the causality condition\footnote{In
    particular this means that $\finc(t)=0$ for all $t\in (-\infty,R)$.} 
  \begin{equation}
    \label{eq:CausalityUi}
    \ui
    \,=\, \dit\ui \,=\, 0 \,, \qquad
    \text{in } \BR \times (-\infty,0) \,, 
  \end{equation}
\end{subequations}
the scattering problem reads 
\begin{subequations}
  \label{eq:WaveEquationU}
  \begin{align}
    \frac{1}{\rho c^2} \dit^2\ut
    - \div\Bigl(\frac{1}{\rho}\nabla\ut\Bigr)
    &\,=\, 0
    &&\text{in } \Rd\times(0,\infty) \,, \\
    \ut(\ph,0) \,=\, \ui(\ph,0) \,,\;
    \dit\ut(\ph,0) &\,=\, \dit\ui(\ph,0)
    &&\text{in } \Rd \,.
  \end{align}
\end{subequations}
As usual we call $\ut$ the total wave, and the scattered wave $\us$ is 
defined as 
\begin{equation*}
  \us(x,t) \,:=\, \ut(x,t) - \ui(x,t) \,, \qquad
  (x,t)\in\Rd\times(0,\infty) \,.
\end{equation*}
Accordingly, 
\begin{subequations}
  \label{eq:WaveEquationUs}
  \begin{equation}
    \label{eq:WaveEquationUsSourceProblem}
    \frac{1}{\rho c^2} \di_t^2 \us
    - \div\Bigl(\frac{1}{\rho}\nabla\us\Bigr)
    \,=\,  \Bigr( 1-\frac{1}{\rho c^2}\Bigr) \di_t^2 \ui
    - \div \Bigl( \Bigl( 1-\frac{1}{\rho}\Bigr)
    \nabla\ui \Bigr) 
    \qquad \text{in } \Rd\times(0,\infty) \,,
  \end{equation}
  and we note that \eqref{eq:CausalityUi} implies that $\us$ is
  causal, i.e.,
  \begin{equation}
    \us(\ph,0) \,=\, \dit\us(\ph,0) \,=\, 0 \qquad \text{in } \Rd \,.
  \end{equation}
\end{subequations}
Our regularity assumptions on $f^i$ imply that 
$\ui|_\BR \in L^2(\R,H^2(\BR))\cap H^1(\R,H^1(\BR))$, and~$\ui|_\BR$ 
is compactly supported with respect to time.
In particular, the right hand side 
of~\eqref{eq:WaveEquationUsSourceProblem} belongs 
to~$L^2(\R;L^2(\BR))$.
It is well-known (as, e.g., outlined in~\cite{KirRie14}) that the
scattering problem~\eqref{eq:WaveEquationUs} possesses a unique
weak solution 
$\us\in C([0,\infty),H^1(\Rd))\cap C^1([0,\infty),L^2(\Rd))$
satisfying 
\begin{multline*}
  \int_0^\infty \int_\Rd \Bigl(
  \frac{1}{\rho} \nabla\us(t)\cdot\nabla \phi(t)
  - \frac{1}{\rho c^2} \di_t\us(t) \di_t \phi(t) 
  \Bigr) \dx \dt \\
  \,=\, \int_0^\infty \int_D \biggl(
  \Bigl(1-\frac{1}{\rho_0}\Bigr) \nabla\ui(t)\cdot\nabla \phi(t)
  - \Bigl(1-\frac{1}{\rho_0 c_0^2}\Bigr) \di_t\ui(t) \di_t \phi(t) 
  \biggr) \dx \dt 
\end{multline*}
for all $\phi \in C^\infty_0([0,\infty),H^1_c(\Rd))$.  
Since the coefficients $c$ and $\rho$ are piecewise constant,
\eqref{eq:WaveEquationU} can equivalently also be written as a
transmission problem for $(\ut|_D,\us|_{\Dplus})$.
As usual, we use in the following the notation $D^-:=D$ for the
interior domain and $D^+:=\Rd\setminus\ol{D}$ for the complementary
exterior domain.  We denote the one-sided trace for $D^\pm$ by 
$\gamma^\pm$, and the normal derivative for $D^\pm$ by $\di_\nu^\pm$.
Therewith, the transmission problem is given by 
\begin{subequations}
  \label{eq:TransmissionProblemTD}
  \begin{align}
    \di_t^2\us - \Delta\us
    &\,=\, 0  
    &&\text{in } D^+\times(0,\infty) \,, \\
    \frac{1}{c_0^2} \di_t^2 \ut - \Delta\ut 
    &\,=\, 0  
    &&\text{in } D^-\times(0,\infty) \,, \\
    \gamma^+\us - \gamma^-\ut
    &\,=\, -\gamma\ui \,,
    &&\text{on } \di D\times(0,\infty) \,,\\ 
    \di_\nu^+\us
    - \frac{1}{\rho_0} \di_\nu^-\ut
    &\,=\, -\di_\nu\ui \,,
    &&\text{on } \di D\times(0,\infty) \,,
  \end{align}
\end{subequations}
with homogeneous initial conditions
$\us(\ph,0) = \dit\us(\ph,0) = 0$ in~$D^+$
and $\ut(\ph,0) = \dit\ut(\ph,0) = 0$ in~$D^-$.

Following \cite{Fri62,Fri64,Fri67}, we are now in a position to
introduce the definition of the far field pattern of the scattered
wave. 
Throughout, we employ, for any Hilbert space $X$ and $r\in\R$, 
the notations
\begin{align*}
    H^r_c(\R; X) 
    &\,:=\, \bigl\{ g \in H^r(\R; X)  
    \;\big|\; \supp g \subs \R \text{ is compact} \bigr\} \,, \\
    H^r_0(0,T;X) 
    &\,:=\, \bigl\{ g|_{(0,T)} 
    \;\big|\; g \in H^r(\R; X) \,,\; 
    g = 0 \text{ in } (-\infty, 0) \bigr\} \,, \\
    H^r_{0,\loc}(0,\infty;X) 
    &\,:=\, \bigl\{ g 
    \;\big|\; g \in H^r_0(0,T; X) \text{ for all } T>0 \bigr\} \,. 
\end{align*}

\begin{lemma}
  \label{lmm:FarfieldExpansion}
  Suppose that $\ui$ is an incident plane wave as 
  in~\eqref{eq:PlaneWave} with $f^i \in H^r_c(\R)$ for some $r>5$, which
  satisfies~\eqref{eq:CausalityUi}. 
  Then the associated scattered field
  $\us|_{D^+} \in
  H^{r-\frac52}_{0,\loc}(0,\infty;H^1(D^+))$ in the exterior of the
  scatterer has the asymptotic behavior 
  \begin{equation*}
    \us(x,|x|+t) 
    \,=\, \frac{1}{|x|} \Bigl( \uinfty(\xi,t) + O(|x|^{-1}) \Bigr) \,, \qquad
    |x| \to \infty \,,
  \end{equation*}
  uniformly in all directions $\xi=x/|x|\in\Sd$ and also in $t\geq 0$
  if $t$ is confined to a finite interval. 
  We call $\uinfty:\Sd\times (0,\infty) \to \R$ the far field pattern
  of $\us$, and we have 
  \begin{equation}
    \label{eq:FarfieldTD}
    \uinfty(\xi,t)
    \,=\, \frac{1}{4\pi} \int_{\di D} \Bigl(
    (\nu_y\cdot\xi) \di_t\us(y,t+\xi\cdot y) 
    - \di_{\nu_y}^+ \us(y,t+\xi\cdot y) 
    \Bigr) \ds_y \,,
    \qquad \xi\in\Sd \,,\; t>0 \,.
  \end{equation}
\end{lemma}

\begin{proof}
  We will see in Theorem~\ref{thm:WellPosednessScatterinProblemTD}
  below that $\ui$ as in~\eqref{eq:PlaneWave} with $f^i \in H^r_c(\R)$
  for some~$r>5$ satisfying~\eqref{eq:CausalityUi} implies 
  that~$\us \in H^{r-\frac52}_0(0,T;H^1(D^+))$ in the exterior 
  of the scatterer for all~$T>0$. 
  By Sobolev embedding this means that $\us$ is twice continuously
  differentiable with respect to the time variable. 
  Using Kirchhoff's representation formula (see, e.g.,
  \cite[p.~97]{BanSay22} or \cite[p.~147]{Mar21}) the scattered
  field~$\us(x,t)$ with $(x,t) \in D^+\times(0,\infty)$ in
  the exterior of the scatterer can be written as the sum of a
  transient double and single layer potential, 
  \begin{equation}
    \label{eq:KirchhoffExtUs}
    \begin{split}
      \us(x,t) 
      &\,=\, \int_{\di D} \biggl(
      \di_{\nu_y}^+\Bigl( \frac{\us(z, t-|x-y|)}{4\pi|x-y|} \Bigr)\bigg|_{z=y} 
      - \frac{1}{4\pi|x-y|} \di_{\nu_y}^+\us(y,t-|x-y|)  
      \biggr) \ds_y \\
      &\,=\, \frac{1}{4\pi} \int_{\di D} \biggl(
      \us(y,t-|x-y|) \di_{\nu_y}^+ \frac{1}{|x-y|}
      - \frac{1}{|x-y|} \di_t\us(y,t-|x-y|) \di_{\nu_y}|x-y| \\
      &\phantom{\,=\, \frac{1}{4\pi} \int_{\di D} \biggl(} 
      - \frac{1}{|x-y|} \di_{\nu_y}^+\us(y,t-|x-y|) \biggr) \ds_y \,.
    \end{split}
  \end{equation}
  Observing that 
  \begin{align*}
    |x-y| 
    &\,=\, |x|-\xi\cdot y + O(|x|^{-1}) \,,
    & |x-y|^{-1}
    &\,=\, |x|^{-1} + O(|x|^{-2}) \,, \\
    \di_{\nu_y}|x-y|
    &\,=\, -\nu_y \cdot \xi + O(|x|^{-1}) \,,
    & \di_{\nu_y}( |x-y|^{-1} )
    &\,=\, \nu_y \cdot  |x|^{-2} \xi + O(|x|^{-3}) \,,
  \end{align*}
  uniformly for all $y\in\di D$, and inserting this into
  \eqref{eq:KirchhoffExtUs}, the result follows from Taylor's theorem.   
\end{proof}

We are interested in the following inverse backscattering problem.
Suppose we are given the far field patterns $\uinfty(-\theta_k,t;\theta_k)$
corresponding to incident plane waves $\ui(\ph,\ph;\theta_k)$
with direction of propagation $\theta_k$ for a finite set of
incident directions $\theta_k\in\Sd$ and on a finite time 
interval~$t\in [0,T]$, can we numerically reconstruct the shape of the 
scatterer~$D$?

\section{Direct scattering problem in the Laplace domain}
\label{sec:ScatteringProblemLD}
We analyze the scattering problem and derive the associated shape 
derivative first in the Laplace domain. 
Following \cite{BanSay22}, we denote by 
\begin{equation*}
  \ghat(s) 
  \,:=\, \Lcal\{g\}(s) 
  \,:=\, \int_{0}^{\infty} \rme^{-st} g(t) \dt \,, \qquad
  s\in\C_+ := \{z\in\C \;|\; \real(z)>0 \} \,,
\end{equation*}
the Laplace transform of a causal function $g:\R\to\R$, assuming that
it exists.  
Given any Hilbert space $X$, the Laplace transform can then be
extended to
\begin{equation*}
  \Lscr'_+(\R;X)
  \,:=\, \big\{ g\in \Dscr'_+(\R;X) \;\big|\; 
  \rme^{-\sigma_g t} g(t) \in \Sscr'_+(\R;X)
  \text{ for some } \sigma_g\geq 0 \big\} \,.
\end{equation*}
Here, $\Dscr'_+(\R;X)$ and $\Sscr'_+(\R;X)$ denotes the space of
causal distributions and of causal tempered distributions with values
in $X$, respectively. 
If $g\in\Lscr'_+(\R;X)$, then its Laplace transform $\Lcal\{g\}(s)$ is 
a holomorphic function of $s$ in the half-plane $\real(s)>\sigma_g$;
see, e.g., \cite[p.~417]{Tre75}. 
In particular, the restricted causal incident plane wave
$\ui|_{\BR} \in H^r_c(\R;H^2(\BR)$, $r \geq 2$, from \eqref{eq:PlaneWave}
has a well-defined Laplace transform, which we denote by~$\uihat|_\BR$.

The Laplace transform of the scattering problem
\eqref{eq:WaveEquationUs} gives, for all $s\in\C_+$, the equation
\begin{equation}
  \label{eq:ScatteringProblemLD}
  \div\Bigl(\frac{1}{\rho}\nabla\ushat\Bigr)
  - \frac{s^2}{\rho c^2} \ushat
  \,=\,  \div \Bigl( \Bigl( 1-\frac{1}{\rho}\Bigr)
  \nabla\uihat \Bigr) - s^2 \Bigr( 1-\frac{1}{\rho c^2}\Bigr) \uihat
  \qquad \text{in } \Rd 
\end{equation}
for the Laplace transform of the scattered field denoted by $\ushat$. 
Also this equation has to be understood in weak sense. 
Since the wave speed $c$ and the density $\rho$ are piecewise
constant, \eqref{eq:ScatteringProblemLD} can equivalently be written
as a transmission problem for $(\vhat^-,\vhat^+) = (\uthat|_{D^-},\ushat|_{D^+})$
that is given by
\begin{subequations}
  \label{eq:TransmissionProblemLD}
  \begin{align}
    \Delta\vhat^+ - s^2\vhat^+
    &\,=\, 0  
    &&\text{in } D^+ \,,\\
    \Delta\vhat^- - \frac{s^2}{c_0^2}\vhat^-
    &\,=\, 0  
    &&\text{in } D^- \,,\\
    \gamma^+\vhat^+ - \gamma^-\vhat^-
    &\,=\, \etahat
    &&\text{on } \di D \,,\\
    \di_\nu^+\vhat^+
    - \frac{1}{\rho_0} \di_\nu^-\vhat^-
    &\,=\, \zetahat
    &&\text{on } \di D \, ,
  \end{align}
\end{subequations}
with $\etahat = - \gamma\uihat$ and 
$\zetahat = -\di_\nu\uihat$. 
In the following, we study the transmission problem 
\eqref{eq:TransmissionProblemLD} in a slightly more 
general form by considering general transmission 
functions $\etahat \in \HhalfdiD$ and $\zetahat \in \HminhalfdiD$.

Following \cite{BamHaD86,BanSay22}, we consider for 
each $s\in\C_+$ the rescaled $H^1$-norm defined by
\begin{equation*}
  \snorm \phihat\snorm_{\Rd,s}^2
  \,:=\, \|\nabla \phihat \|_{L^2(\Rd)}^2 + \|s \phihat \|_{L^2(\Rd)}^2 \,,
  \qquad \phihat\in H^1(\Rd) \,.
\end{equation*}
We note that
\begin{equation}
  \label{eq:EquivalenceSnorm}
  \min(1,|s|) \|\phihat\|_{H^1(\Rd)}
  \,\leq\, \snorm \phihat\snorm_{\Rd,s}
  \,\leq\, \max(1,|s|) \|\phihat\|_{H^1(\Rd)} \,;
\end{equation}
see, e.g., \cite[p.~86]{BanSay22}.
In the next lemma we establish existence and uniqueness of solutions
to~\eqref{eq:TransmissionProblemLD} and prove a basic energy estimate. 

\begin{lemma}
  \label{lmm:StabilityUshat}
  For any $s\in\C$ with $\real(s)>0$ the transmission problem
  \eqref{eq:TransmissionProblemLD} with arbitrary functions~$\etahat \in \HhalfdiD$ 
  and~$\zetahat \in \HminhalfdiD$ has a unique weak
  solution~${(\vhat^-,\vhat^+) \in H^1(D^-)\times H^1(D^+)}$, which satisfies 
  \begin{equation}
    \label{eq:StabilityUshat}
    \snorm\vhat^-\snorm_{D^-,s} + \snorm\vhat^+\snorm_{D^+,s}
    \,\leq\, C_\sigma \frac{|s|^{\frac32}}{\real(s)} 
    \Bigl( \|\zetahat\|_{\HminhalfdiD}^{2} 
    + \|\etahat\|_{\HhalfdiD}^{2} \Bigr)^{\frac12} \,,
    \qquad \real(s) \geq \sigma > 0 \,.
  \end{equation}
\end{lemma}

\begin{proof}
  The weak formulation we focus on is to find $\zhat \in H^1(\R^3)$ satisfying
  \begin{equation}
    \label{eq:ScatteringProblemzhat}
    a(\zhat, \phihat) 
    \,=\, \ell(\phihat) \qquad \text{for all } \phihat \in H^1(\R^3) \,,
  \end{equation}
  where
  \begin{align*}
    a(\zhat, \phihat) 
    &\,:=\, \frac{1}{\rho_0} \int_{D^-} \nabla \zhat \cdot \ol{\nabla \phihat} 
      + \frac{s^2}{c_0^2} \zhat \, \ol{\phihat} \dx 
      + \int_{D^+} \nabla \zhat \cdot \ol{\nabla \phihat} 
      + s^2 \zhat \, \ol{\phihat} \dx \,, \\
    \ell(\phihat) 
    &\,=\, - \int_{\di D} \zetahat \, \ol{\gamma \phihat}\ds 
      - \biggl( \int_{D^+} \nabla (\Lambda \etahat) \cdot \ol{\nabla \phihat} 
      + s^2 (\Lambda \etahat) \, \ol{\phihat} \dx \biggr) \,,
  \end{align*}
  where $\Lambda \etahat$ is any lifting of $\etahat$ into $D^+$.
  In this proof we pick $\Lambda \etahat = \what$, where $\what$ is the solution to
  \begin{align*}
    -\Delta \what + |s|^2 \what 
    &\,=\, 0 
    && \text{in } D^+ \,, \\
    \gamma^+\what &\,=\, \etahat 
    && \text{on } \di D \,,
  \end{align*}
  which, by \cite[Lmm.~4.4]{BanSay22}, satisfies
  \begin{equation}
    \label{eq:bounds12}
    \snorm \what \snorm_{D^+,s} 
    \,\leq\, C_\sigma |s|^{\frac{1}{2}}\| \etahat \|_{\HhalfdiD} \,,
    \qquad \real(s) \geq \sigma > 0 \,.
  \end{equation}
  Once $\zhat \in H^1(\R^3)$ is shown to be uniquely determined we define
  \begin{equation}
    \label{eq:uhatdecomp}
    (\vhat^-, \vhat^+) 
    \,=\, (\zhat|_{D^-}, \zhat|_{D^+}) + (0, \Lambda \etahat) \,,
  \end{equation}
  which is then the unique weak solution of \eqref{eq:TransmissionProblemLD}.
  Since 
  \begin{equation}
    \label{eq:coerc-a}
    \begin{split}
      |s||a(\zhat, \zhat) | \, &\geq \,
      \real\biggl( \ol{s} 
      \biggl(\frac{1}{\rho_0} \int_{D^-} |\nabla \zhat|^2 + \frac{s^2}{c_0^2} |\zhat|^2  \dx + \int_{D^+}|\nabla \zhat|^2 + s^2 |\zhat|^2 \dx
      \biggr)
      \biggr) \\
      & \geq \, \min\Bigl(\frac{1}{\rho_0}, \frac{1}{\rho_0 c_0^2} \Bigr) \real(s) \Bigl(\snorm\zhat \snorm_{D^-,s}^2 + \snorm \zhat \snorm_{D^+,s}^2\Bigr) \,,
    \end{split}
  \end{equation}
  the sesquilinear form $a$ is coercive for all $s\in\C_+$.
  Hence, existence and uniqueness of a weak solution~$\zhat\in H^1(\Rd)$ of \eqref{eq:ScatteringProblemzhat} follows from
  the lemma of Lax-Milgram.
  Moreover, combining \eqref{eq:coerc-a} with the weak formulation \eqref{eq:ScatteringProblemzhat} for $\phihat = \zhat$, we obtain using
  the Cauchy-Schwarz inequality together with the bound \eqref{eq:bounds12} that 
  \begin{equation*}
    \begin{split}
      \Bigl(\snorm\zhat \snorm_{D^-,s} + \snorm \zhat \snorm_{D^+,s}\Bigr)^2
      &\,\leq\, C \frac{|s|}{\real (s)} |\ell(\zhat)| \\
      &\,\leq\, C \frac{|s|}{\real (s)} 
      \Bigl( \| \zetahat \|_{\HminhalfdiD}
      \| \gamma\zhat\|_{\HhalfdiD} 
      + \snorm \Lambda \etahat \snorm_{D^+,s}
      \snorm \zhat \snorm_{D^+,s}\Bigr) \\
      &\,\leq\, C \frac{|s|^{\frac{3}{2}}}{\real(s)}
      \Bigl(\| \zetahat \|_{\HminhalfdiD} 
      + \| \etahat\|_{\HhalfdiD}\Bigr)
      \Bigl( \snorm \zhat \snorm_{D^-,s} 
        + \snorm \zhat \snorm_{D^+,s} \Bigr) \,.
    \end{split}
  \end{equation*}

  Now we use the representation \eqref{eq:uhatdecomp} together with \eqref{eq:bounds12} once more to see that
  \begin{equation*}
    \begin{split}
      \snorm\vhat^-\snorm_{D^-,s} + \snorm\vhat^+\snorm_{D^+,s}
      &\,\leq\,
      \snorm\zhat\snorm_{D^-,s}
      + \snorm\zhat\snorm_{D^+,s}
      + \snorm \Lambda \etahat \snorm_{D^+,s} \\
      &\,\leq\, C_\sigma \frac{|s|^{\frac{3}{2}}}{\real (s)}
      \Bigl(\| \zetahat \|_{\HminhalfdiD}^{2} 
      + \| \etahat\|_{\HhalfdiD}^{2} \Bigr)^{\frac12} \,,
      \qquad \real(s) \geq \sigma > 0 \,.
    \end{split}
  \end{equation*}
\end{proof}

To establish an explicit representation of the solution
of the transmission problem~\eqref{eq:TransmissionProblemLD} 
we rely on integral equations. 
Accordingly, we recall the definitions of the single layer potential
and the double layer potential for the partial differential operator
$\Pcal := -\Delta + s^2$ with complex $s\in\C_+$,
\begin{subequations}
  \label{eq:DefSLDL}
  \begin{align}
    \SL(s):\; \HminhalfdiD \to H^{1}(\Rd) \,,
    & \qquad (\SL(s)\varphi)(x) 
      \,:=\, \int_{\di D} \varphi(y) \Phi_s(x-y)  \ds_y \,,\\
    \DL(s):\; \HhalfdiD \to H^{1}(D^\pm) \,,
    & \qquad (\DL(s) \psi)(x) 
      \,:=\, \int_{\di D} \psi(y) \di_{\nu_y} \Phi_s(x-y) \ds_y \,.
  \end{align}
\end{subequations}
Here, 
\begin{equation*}
  \Phi_s(x) 
  \,:=\, \frac{\rme^{-s |x|}}{4\pi |x|} \,, \qquad 
  x\in\Rd\setminus\{0\} \,,
\end{equation*}
is the fundamental solution for $\Pcal$ in $\Rd$.
We have the following pointwise estimates for the single and double
layer potential away from $D$. 

\begin{lemma}
  \label{lmm:PointwiseEstSLDLLD}
  For $z\in D^+$ with $\dist(z,\di D)\geq \delta> 0$ and $s\in\C_+$ with
  $\real(s) \geq \sigma > 0$, we have
  \begin{subequations}
    \label{eq:PointwiseEstSLDLLD}
    \begin{align}
      |(\SL(s)\varphi)(z)|
      &\,\leq\, C_{\sigma,\delta} |s|
        \frac{\rme^{-\dist(z,\di D)\real(s)}}{\dist(z,\di D)}
        \|\varphi\|_{\HminhalfdiD} \,,
      &&\varphi\in\HminhalfdiD \,, \label{eq:PointwiseEstSLLD}\\
      |(\DL(s)\psi)(z)|
      &\,\leq\, C_{\sigma,\delta} |s|^{\frac{3}{2}}
        \frac{\rme^{-\dist(z,\di D)\real(s)}}{\dist(z,\di D)}
        \|\psi\|_{\HhalfdiD} \,,
      &&\psi\in\HhalfdiD \,. \label{eq:PointwiseEstDLLD}
    \end{align}
  \end{subequations}
\end{lemma}

\begin{proof}
  The inequality \eqref{eq:PointwiseEstSLLD} can be shown as
  in~\cite[Lmm.~7]{BanLubMel11}. 
  To establish \eqref{eq:PointwiseEstDLLD}, we follow the same strategy and use
  \cite[Lmm.~4.5]{BanSay22} to obtain
  \begin{equation*}
    \begin{split}
      |(\DL(s)\psi)(z)|
      &\,\leq\, \Bigl\|\di_\nu
      \frac{\rme^{-s|z-\cdot|}}{4\pi|z-\cdot|} \Bigr\|_{\HminhalfdiD}
      \|\psi\|_{\HhalfdiD} \\
      &\,\leq\, C \max(1,|s|^{-\frac12}) |s|^{\frac12}
      \Bigsnorm\frac{\rme^{-s|z-\cdot|}}{4\pi|z-\cdot|} \Bigsnorm_{D,s}
      \|\psi\|_{\HhalfdiD} \\
      &\,\leq\, C_{\sigma,\delta} |s|^{\frac32} 
      \frac{\rme^{-\dist(z,\di D)\real(s)}}{\dist(z,\di D)} \|\psi\|_{\HhalfdiD} \,.
    \end{split}
  \end{equation*}
\end{proof}

Using Green's representation theorem (see, e.g.,
\cite[Thm.~6.10]{McL00}), we find that the weak 
solution~$(\vhat^-,\vhat^+)\in H^1(D^-)\times H^1(D^+)$
of~\eqref{eq:TransmissionProblemLD} (or equivalently
of~\eqref{eq:ScatteringProblemLD}) satisfies
\begin{subequations}\label{eq:representationvpvm}
    \begin{align}
  \vhat^+
  &\,=\, - \SL(s)(\di_\nu^+\vhat^+) + \DL(s)(\gamma^+\vhat^+) 
  &&\text{in } D^+ \,, \label{eq:v+repGreen}\\
  \vhat^- 
  &\,=\, \SL(s_0)(\di_\nu^-\vhat^-) - \DL(s_0)(\gamma^-\vhat^-) 
  &&\text{in } D^- \,,
\end{align}
\end{subequations}
where $s_0 := s/c_0$.
Applying the jump relations for the single and double layer potential on
$\di D$ and the associated boundary integral operators 
\begin{align*}
  \SLop(s): \, \HminhalfdiD \to \HhalfdiD \,,
  && (\SLop(s)\varphi)(x) 
  &\,:=\, \int_{\di D} \varphi(y) \Phi_s(x-y) \ds_y \,, \\
  \DLop(s): \, \HhalfdiD \to \HhalfdiD \,,
  && (\DLop(s)\psi)(x) 
  &\,:=\, \int_{\di D} \psi(y) \di_{\nu_y}\Phi_s(x-y) \ds_y \,, \\
  \adjDLop(s): \, \HminhalfdiD \to \HminhalfdiD \,,
  && (\adjDLop(s)\varphi)(x) 
  &\,:=\, \int_{\di D} \varphi(y) \di_{\nu_x}\Phi_s(x-y) \ds_y \,, \\ 
  \hypSingop(s): \, \HhalfdiD \to \HminhalfdiD \,,
  && (\hypSingop(s)\varphi)(x) 
  &\,:=\, \di_{\nu_x} \int_{\di D} \varphi(y) \di_{\nu_y}\Phi_s(x-y) \ds_y \,,
\end{align*}
(see, e.g., \cite[Thm.~6.11]{McL00}) it follows immediately that 
$(\frac1s\di_\nu^+\vhat^+,-\gamma^+\vhat^+)\in \HminhalfdiD\times\HhalfdiD$
satisfies the integral equation 
\begin{equation}
  \label{eq:BoundaryIntegralEqn_ScatteringProblem}
  \begin{bmatrix}
    \rho_0 s \SLop(s_0) + s\SLop(s) 
    & \DLop(s_0)+\DLop(s) \\
    - \adjDLop(s_0) - \adjDLop(s)
    & \frac{1}{\rho_0 s}\hypSingop(s_0)
    + \frac{1}{s}\hypSingop(s) 
  \end{bmatrix}
  \begin{bmatrix}
    \frac{1}{s}\di_\nu^+\vhat^+ \\
    -\gamma^+\vhat^+
  \end{bmatrix}
  \,=\,
  \begin{bmatrix}
    -\rho_0 s \SLop(s_0) & 
    -\DLop(s_0)-\frac{1}{2}I \\
    \adjDLop(s_0) - \frac{1}{2} I & 
    -\frac{1}{\rho_0s} \hypSingop(s_0)
  \end{bmatrix}
  \begin{bmatrix}
    -\frac{1}{s}\zetahat \\
    \etahat
  \end{bmatrix} 
\end{equation}
on $\di D$.

We define $\Xcal:=\HminhalfdiD\times \HhalfdiD$ with the norm given by
\begin{equation*}
  \|(\varphi,\psi)\|_\Xcal^2 
  \,:=\, \|\varphi\|_{\HminhalfdiD}^2 + \|\psi\|_{\HhalfdiD}^2 \,,
  \qquad (\varphi,\psi) \in \Xcal \,,
\end{equation*}
and we write $\Xcal'=\HhalfdiD\times \HminhalfdiD$ for the associated dual
space. 
The corresponding anti-dual form is defined by
\begin{equation*}
  \langle (\zeta,\eta) , (\varphi,\psi) \rangle_{\Xcal',\Xcal}
  \,:=\, \langle \ol{\varphi}, \ol{\zeta} \rangle_{\HminhalfdiD,\HhalfdiD}
  + \langle \eta, \psi \rangle_{\HminhalfdiD,\HhalfdiD}
\end{equation*}
for $\zeta,\varphi \in \HminhalfdiD$ and $\eta,\psi \in \HhalfdiD$.  
In Appendix~\ref{app:A} we show that for any $s\in\C_+$ the boundary
integral operator $\Ccal(s): \Xcal \to \Xcal'$ defined by 
\begin{equation}
  \label{eq:DefOpCcal}
  \Ccal(s) 
  \,:=\, \begin{bmatrix}
    \rho_0 s \SLop(s_0) + s\SLop(s) 
    & \DLop(s_0)+\DLop(s) \\
    - \adjDLop(s_0) - \adjDLop(s)
    & \frac{1}{\rho_0 s}\hypSingop(s_0)
    + \frac{1}{s}\hypSingop(s) 
  \end{bmatrix}
\end{equation}
has a bounded inverse $\Ccal^{-1}(s): \Xcal'\to \Xcal$, which satisfies
\begin{equation}
  \label{eq:EstInverseOpCcal}
  \| \Ccal^{-1}(s) \|_{\Xcal \leftarrow \Xcal'} 
  \,\leq\, C_\sigma \frac{|s|^2}{\real(s)} \,, \qquad 
  \real(s)\geq\sigma>0 \,.
\end{equation}
Moreover, let for any $s\in\C_+$ the boundary integral operator
$\Acal: \Xcal\to \Xcal'$ be defined by
\begin{equation}
  \label{eq:DefOpAcal}
  \Acal(s) 
  \,:=\, \begin{bmatrix}
    -\rho_0 s \SLop(s_0) & 
    -\DLop(s_0)-\frac{1}{2}I \\
    \adjDLop(s_0) - \frac{1}{2} I & 
    -\frac{1}{\rho_0 s} \hypSingop(s_0)
  \end{bmatrix} \,.
\end{equation}
Then we obtain from
\eqref{eq:BoundaryIntegralEqn_ScatteringProblem} that the solution
$(\vhat^-,\vhat^+)\in H^1(D^-)\times H^1(D^+)$ of the
transmission problem~\eqref{eq:TransmissionProblemLD} 
with $\etahat \in \HhalfdiD$ and $\zetahat \in \HminhalfdiD$ can be written as
\begin{subequations}
  \label{eq:ExplRepSolnTransmissionProblemLD}
  \begin{align}
    \vhat^+
    &\,=\, - \begin{bmatrix}
      s\SL(s) & \DL(s)
    \end{bmatrix} \Ccal^{-1}(s) \Acal(s) \begin{bmatrix}
      -\frac{1}{s}\zetahat\\
      \etahat
    \end{bmatrix} 
    &&\text{in } D^+ \,, \label{eq:vprep2}\\
    \vhat^-
    &\,=\, \begin{bmatrix}
      \rho_0 s \SL(s_0) & \DL(s_0) 
    \end{bmatrix} \biggl(
                          \begin{bmatrix}
                            I & 0 \\ 0 & I
                          \end{bmatrix}
                                         + \Ccal^{-1}(s) \Acal(s) \biggr)
                                         \begin{bmatrix}
                                           -\frac{1}{s}\zetahat \\
                                            \etahat
                                         \end{bmatrix} 
    &&\text{in } D^- \,.
  \end{align}  
\end{subequations}

Applying the representation for $\vhat^+$ in \eqref{eq:v+repGreen}, the 
bounds from Lemma~\ref{lmm:PointwiseEstSLDLLD} and the estimate in 
Lemma~\ref{lmm:StabilityUshat} yields the next corollary, which gives a 
pointwise estimate for solutions to \eqref{eq:TransmissionProblemLD}.

\begin{corollary}
\label{cor:LDpwbound}
    Let $s\in \C_+$ with $\real(s)\geq\sigma>0$ and let $\vhat^+$ be as 
    in~\eqref{eq:vprep2} with $\etahat \in \HhalfdiD$ and~$\zetahat \in \HminhalfdiD$.
    Then, for any fixed $z\in D^+$ with $\dist(z,\di D)\geq \delta> 0$ there holds
    \begin{equation}
    \label{eq:vhatboundpwLD}
        |\vhat^+(z)| 
        \,\leq\, C_{\sigma,\delta} |s|^3 
        \frac{\rme^{-\dist(z,\di D)\real(s)}}{\dist(z,\di D)} 
        \Bigl( \|\zetahat\|_{\HminhalfdiD}^2
        + \|\etahat\|_{\HhalfdiD}^2 \Bigr)^{\frac12} \,.
    \end{equation}
\end{corollary}

Regularity results for elliptic equations (see, e.g.,
\cite[Thm.~4.20]{McL00}) imply that, for every $s\in\C_+$, the weak
solution $(\vhat^-,\vhat^+)\in H^1(D^-)\times H^1(D^+)$ of 
  \eqref{eq:TransmissionProblemLD} satisfies $\vhat^- \in H^2(D^-)$ and 
  $\vhat^+ \in H^2(D^+)$.
In the following proposition we establish an associated estimate that
is explicit in terms of the complex parameter $s$.
This estimate will be required in the derivation of the domain
derivative in the Laplace domain in
Section~\ref{sec:DomainDerivativeLD}, and the explicit dependence on
$s$ will be required when translating this result back
into the time domain in Section~\ref{sec:TimeDomain} below. 

\begin{proposition}
  \label{pro:ScatteringProblemLD_H2estimate}
  Suppose $s\in\C_+$ with $\real(s)\geq\sigma>0$, and
  let~$(\vhat^-,\vhat^+) \in H^1(D^-)\times H^1(D^+)$ be
  the unique weak solution to~\eqref{eq:TransmissionProblemLD} 
  with $\etahat \in \HthreehalfdiD$ and~$\zetahat \in \HhalfdiD$. 
  Then, 
  \begin{equation}
    \label{eq:ScatteringProblemLD_H2estimate}
    \|\vhat^-\|_{H^2(D^-)} + \|\vhat^+\|_{H^2(D^+)} 
    \,\leq\, C_\sigma \frac{|s|^\frac{5}{2}}{\real(s)^{\frac{1}{2}}}
    \Bigl( \|\zetahat\|_{\HhalfdiD}^2
    + \|\etahat\|_{\HthreehalfdiD}^2 \Bigr)^{\frac12} 
  \end{equation}
  for  $s\in\C_+$ with $\real(s)\geq\sigma>0$.
\end{proposition}

\begin{proof}
  Proceeding as in the proof of \cite[Thm.~4.20]{McL00}, but tracking
  the dependency on $s$, we find that, for $\real(s)\geq\sigma>0$,
  \begin{equation*}
    \|\vhat^-\|_{H^2(D^-)} + \|\vhat^+\|_{H^2(D^+)}
    \,\leq\, C_\sigma |s| \Bigl(
      \snorm\vhat^-\snorm_{D^-,s} 
      + \snorm\vhat^+\snorm_{D^+,s} \Bigr)
      + C_\sigma |s|^2 \Bigl( \| \zetahat \|_{\HhalfdiD} 
      + \|\etahat\|_{\HthreehalfdiD} \Bigr) \,.
  \end{equation*}
  Substituting \eqref{eq:StabilityUshat} on the right hand side of this
  estimate gives~\eqref{eq:ScatteringProblemLD_H2estimate}. 
\end{proof}

Next we discuss the Laplace transform of the far field pattern from
Lemma~\ref{lmm:FarfieldExpansion}. 
Since the function $\uinfty$ in \eqref{eq:FarfieldTD} is not causal,
we introduce a shifted version $\uinfty_R\in L^2(\Sd\times(0,\infty))$ by
\begin{equation}
  \label{eq:ShiftedFarfieldPatternTD}
  \uinfty_R(\xi, t)
  \,:=\, \uinfty(\xi,t-R) \,, \qquad t\in(0,\infty) \,.
\end{equation}
Here $R>0$ is as in the causality condition~\eqref{eq:CausalityUi}. 
Then $\uinfty_R$ is causal and its Laplace transform satisfies
\begin{equation}
  \label{eq:ShiftedFarfieldPatternLD}
  \begin{split}
    \uinftyhat_R(\xi,s)
    &\,=\, \frac{1}{4\pi} \int_0^\infty \rme^{-st}
    \int_{\di D} \Bigl(
    (\nu_y\cdot\xi) \di_t(\gamma^+\us)(y,t-R+\xi\cdot y) 
    - (\di_{\nu_y}^+\us)(y,t-R+\xi\cdot y) 
    \Bigr) \ds_y \dt \\
    &\,=\, \frac{1}{4\pi} \int_{\di D} \rme^{-s(R-\xi\cdot y)}
    \int_0^\infty \rme^{-s \tau} \bigl(
    (\nu_y\cdot\xi) \di_t(\gamma^+\us)(y,\tau) 
    - (\di_{\nu_y}^+ \us)(y,\tau) \bigr) \dtau \ds_y \\
    &\,=\, \frac{1}{4\pi} \rme^{-sR} \int_{\di D}
    \rme^{s \xi\cdot y} 
    \bigl(
    (\nu_y\cdot\xi) s(\gamma^+\ushat)(y,s) 
    - (\di_{\nu_y}^+ \ushat)(y,s) \bigr) \ds_y \,.
  \end{split}
\end{equation}
In the second step we used the causality of the scattered field~$\us$.
We denote by
\begin{align*}
  \SL^\infty_R(s):\; \HminhalfdiD \to L^2(\Sd) \,,
  & \qquad (\SL^\infty_R(s)\varphi)(x) 
    \,:=\, \frac{1}{4\pi} \rme^{-sR}
    \int_{\di D} \varphi(y) \rme^{s \xi\cdot y} \ds_y \,,\\
  \DL^\infty_R(s):\; \HhalfdiD \to L^2(\Sd) \,,
  & \qquad (\DL^\infty_R(s) \psi)(x) 
    \,:=\, \frac{1}{4\pi} \rme^{-sR}
    \int_{\di D} \psi(y) (\nu_y\cdot\xi) s \rme^{s \xi\cdot y} \ds_y \,,
\end{align*}
the shifted far field patterns of the single and double layer
potentials from \eqref{eq:DefSLDL}.
Then we obtain from \eqref{eq:ExplRepSolnTransmissionProblemLD} the
explicit representation
\begin{equation}
  \label{eq:ExplRepFarfieldTransmissionProblemLD}
  \uinftyhat_R
  \,=\, - \begin{bmatrix}
    s\SL^\infty_R(s) & \DL^\infty_R(s)
  \end{bmatrix} \Ccal^{-1}(s) \Acal(s) \begin{bmatrix}
    \frac{1}{s}\di_\nu\uihat \\
    -\gamma\uihat 
  \end{bmatrix} \qquad \text{on } \Sd \,.
\end{equation}
This is also the representation and the associated integral
equation that we will use for the numerical simulation of solutions of
the direct scattering problem in Section~\ref{sec:NumericalExamples}. 

In order to translate this representation into the time-domain in
Section~\ref{sec:TimeDomain} below, we require the following pointwise
estimate. 

\begin{proposition}
  \label{pro:EstFarfieldPatternPointwiseLD}
  Suppose $s\in\C_+$, and let 
  \begin{equation}\label{eq:vinftyRrepLD}
    \vinftyhat_R
      \,=\, - \begin{bmatrix}
        s\SL^\infty_R(s) & \DL^\infty_R(s)
      \end{bmatrix} \Ccal^{-1}(s) \Acal(s) \begin{bmatrix}
        - \frac{1}{s}\zetahat \\
        \etahat
      \end{bmatrix} \qquad \text{on } \Sd \,,
  \end{equation}
  be the shifted far field pattern corresponding to the unique 
  solution~$(\vhat^-,\vhat^+) \in H^1(D^-)\times H^1(D^+)$ of the 
  transmission problem~\eqref{eq:TransmissionProblemLD} with 
  $\etahat \in \HhalfdiD$ and~$\zetahat \in \HminhalfdiD$.
  Denoting
  \begin{equation*}
    \delta_D(\xi)
    \,:=\, \sup_{x\in D}\, x\cdot \xi \,, \qquad \xi\in\Sd \,,
  \end{equation*}
  we have the estimate
  \begin{equation}
    \label{eq:EstFarfieldPatternPointwiseLD}
    | \vinftyhat_R(\xi) |
    \,\leq\, C_\sigma \rme^{-\real(s)(R-\delta_D(\xi))}
    \frac{|s|^{\frac52}}{\real(s)}
    \Bigl( \|\zetahat\|_{\HminhalfdiD}^2
    + \|\etahat\|_{\HhalfdiD}^2 \Bigr)^{\frac{1}{2}} \,,
    \qquad \xi\in\Sd \,.
  \end{equation}
\end{proposition}

\begin{proof}  
  From~\eqref{eq:ShiftedFarfieldPatternLD} we immediately obtain that 
  \begin{equation*}
    | \vinftyhat_R(\xi) |
    \,\leq\, C \rme^{-\real(s)(R-\delta_D(\xi))}
    \Bigl( |s| \|\gamma^+\vhat^+\|_{\HhalfdiD}
    + \| \di_{\nu_y}^+\vhat^+\|_{\HminhalfdiD} \Bigr) \,.
  \end{equation*}
  Combining this with the continuity of the trace operator
  $\gamma:H^1(D^+)\to\HhalfdiD$ and
  \cite[Lmm.~4.5]{BanSay22} gives 
  \begin{equation*}
    \begin{split}
      | \vinftyhat_R(\xi,s) |
      &\,\leq\, C \rme^{-\real(s)(R-\delta_D(\xi))}
      \Bigl( |s| \| \vhat^+ \|_{H^1(D^+)}
      + \max(1,|s|^{-\frac12}) |s|^{\frac12} \snorm\vhat^+\snorm_{D^+,s}
      \Bigr) \,.
    \end{split}
  \end{equation*}
  Applying~\eqref{eq:EquivalenceSnorm} and~\eqref{eq:StabilityUshat}
  we obtain
  \begin{equation*}
    \begin{split}
      | \vinftyhat_R(\xi,s) |
      &\,\leq\, C_\sigma \rme^{-\real(s)(R-\delta_D(\xi))}
      \frac{|s|^{\frac52}}{\real(s)} 
      \Bigl( \|\zetahat\|_{\HminhalfdiD}^2
      + \|\etahat\|_{\HhalfdiD}^2 \Bigr)^{\frac{1}{2}} \,.
    \end{split}
  \end{equation*}
  This ends the proof.  
\end{proof}

\section{Domain derivative in the Laplace domain}
\label{sec:DomainDerivativeLD}
The domain derivative for the transmission problem for the Helmholtz
equation with real wave numbers, i.e., for
\eqref{eq:TransmissionProblemLD} with $s\in\rmi\R$, has been
established in~\cite{Het95}. 
This result including its proof can be transferred to our setting
$s\in\C_+$ with minor modifications. 
However, in contrast to \cite{Het95} we require explicit knowledge on
the dependence of the constants in the associated estimates in terms of 
the complex parameter $s$ in order to obtain corresponding results for the
time-dependent problem via the inverse Laplace transform. 
These estimates can be shown as worked out in detail in the proof
of~\cite[Prop.~2.3]{KnoNic26} for the sound-soft scattering problem.
Since no new ideas are required, we just comment on the main steps of
the proof and refer the reader to~\cite{Het95} and~\cite{KnoNic26} for
details.  

We consider deformations of the $C^2$-boundary $\di D$ of the obstacle
by means of a sufficiently small vector field $h\in C^1(\di D,\Rd)$
and accordingly we denote by
$\di D_h := {\{ y\in\Rd \;|\; y=x+h(x) \,,\; x\in\di D \}}$
the boundary of a perturbed domain $D_h\subs\Rd$.
We always assume that ${\|h\|_{C^1(\di D)} < h_0}$ for
some~$h_0>0$ small enough such that $D_h \subset\subset \BR$ and
$\di D_h$ is a $C^1$ domain. 
Given an observation point $z\in\Rd\setminus\ol{\BR}$, we study the
non-linear operator
\begin{equation}
  \label{eq:OperatorFhatD}
  \Fhat_D :\; \Dcal(\Fhat_D) \subset C^1(\di D,\Rd) \to \C \,,
  \qquad \Fhat_D(h) \,:=\, \ushat_h(z) \,,
\end{equation}
where $(\uthat_h,\ushat_h)\in H^1(D_h)\times H^1(\Rd\setminus\ol{D_h})$
denotes the solution of the transmission
problem~\eqref{eq:TransmissionProblemLD} with~$D$ replaced by $D_h$
and with~$\etahat = - \gamma\uihat|_{\di D_h}$ and 
$\zetahat = -\di_\nu\uihat|_{\di D_h}$. 
Here,
\begin{equation}
    \label{eq:DefDomainFhatD}
    \Dcal(\Fhat_D) 
    \,:=\, \bigl\{ h\in C^1(\di D,\Rd) \;\big|\; \|h\|_{C^1(\di D)}<h_0 \bigr\} \,.
\end{equation}

\begin{proposition}
  \label{pro:FrechetDerivativeLD}
  The operator $\Fhat_D$ from \eqref{eq:OperatorFhatD} is Fr\'echet
  differentiable at zero.
  Denoting by $(\uthat,\ushat) \in H^1(D^-)\times H^1(D^+)$ the
  solution to the transmission
  problem~\eqref{eq:TransmissionProblemLD} with $\etahat = - \gamma\uihat$ 
  and $\zetahat = -\di_\nu\uihat$, the Fr\'echet derivative
  is given by $\Fhat_D'(0)h=\uthat_h'(z)$ for all $h\in C^1(\di D,\Rd)$, where
  $\uthat_h'\in H^1(\Rd\setminus\di D)$ is the unique solution of the
  transmission problem
  \begin{subequations}
    \label{eq:TransmissionProblemFDLD}
    \begin{align}
      \Delta\uthat_h'- s^2\uthat_h'
      &\,=\, 0  
      &&\text{in } D^+ \,,\\
      \Delta\uthat_h' - \frac{s^2}{c_0^2} \uthat_h'
      &\,=\, 0  
      &&\text{in } D^- \,,\\
      \gamma^+\uthat_h' - \gamma^-\uthat_h'
      &\,=\, h_\nu \Bigl( 1 - \frac{1}{\rho_0} \Bigr) \di_\nu^-\uthat
      &&\text{on } \di D \,,\label{eq:TransmissionProblemFDLDc}\\
      \di_\nu^+\uthat_h'
      - \frac{1}{\rho_0} \di_\nu^-\uthat_h'
      &\,=\, - s^2  h_\nu \Bigl( 1 - \frac{1}{\rho_0 c_0^2} \Bigr) \gamma^{-}\uthat
        + \sdiv \Bigl( h_\nu \Bigl( 1-\frac{1}{\rho_0} \Bigr) \sgrad
        \gamma^{-}\uthat \Bigr)
      &&\text{on } \di D \,. \label{eq:TransmissionProblemFDLDd}
    \end{align}
  \end{subequations}
  Here $h_\nu:=h\cdot\nu$, and $\sdiv$ and $\sgrad$ denote the surface
  divergence and the surface gradient on $\di D$, respectively.
  The function $\uthat_h'$ is called the domain derivative of the
  scattered field in the Laplace domain. 
  We have
  \begin{equation}
    \label{eq:EstFrechetDerivativeLDpointwise}
    \bigl| \Fhat_D(h) - \Fhat_D(0) - \Fhat_D'(0)h \bigr|
    \,\leq\, C_\sigma \frac{\rme^{-\dist(z,\di\BR) \real(s)}}{\dist(z,\di\BR)}
    \frac{|s|^{\frac{9}{2}}}{\real(s)^3} \|\uihat\|_{H^1(\BRtilde\setminus\ol{\BR})}
    \|h\|_{C^1(\di D)}^2     
  \end{equation}
  for all $s\in\C_+$ with $\real(s)\geq \sigma>0$ and for any $\Rtilde>R$.   
\end{proposition}

\begin{proof}
  Existence and uniqueness of a solution
  $\uthat_h' \in H^1(\Rd\setminus\di D)$ to~\eqref{eq:TransmissionProblemFDLD} 
  can again be shown by applying the Lemma of Lax-Milgram to the weak formulation of this
  transmission problem.
  Below we will give an alternative proof using integral equation
  techniques. 
  
  Let $\chi\in C^\infty(\Rd)$ be a cut-off function with $\chi=1$ in
  $\BR$ and $\chi=0$ in $\Rd\setminus\ol{\BRtilde}$.
  Accordingly, we define $\what := \ushat + \chi\uihat \in H^1(\Rd)$. 
  Considering the weak formulation of \eqref{eq:ScatteringProblemLD}, it
  follows immediately as in the proof of \cite[Pro.~2.3]{KnoNic26} that
  \begin{equation*}
    \snorm \what \snorm_{\Rd,s}
    \,\leq\, C \frac{|s|}{\real(s)}
    \|\uihat\|_{H^1(\BRtilde\setminus\ol{\BR})} \,.
  \end{equation*}
  
  Denoting by $\ushat_h \in H^1(\Rd)$ the solution to
  \eqref{eq:ScatteringProblemLD} with $D$ replaced by $D_h$ we define
  ${\what_h := \ushat_h + \chi\uihat}$. 
  Considering an extension of the boundary perturbation $h$ to $\BR$
  that is supported in a neighborhood of $\di D$, denoted again by
  $h$, we define a diffeomorphism $\varphi: \BR \to \BR$
  by~$\varphi(x) := x+h(x)$. 
  This can be used to define the transformed field
  $\wtilde_h := \what_h\circ\varphi$.
  It follows by combining the analysis from~\cite{Het95}
  with~\cite{KnoNic26} that  
  \begin{equation*}
    \snorm \wtilde_h -\what \snorm_{\Rd,s}
    \,\leq\, C \frac{|s|^2}{\real(s)^2}
    \|\uihat\|_{H^1(\BRtilde\setminus\ol{\BR})}
    \| h \|_{C^1(\di D)} \,.
  \end{equation*}
  Moreover, it can be shown, following the arguments
  in~\cite{Het95,KnoNic26} that 
  $\What := \uhat_h' + h\cdot\nabla\what$ fulfills
  \begin{equation}
    \label{eq:EstMaterialDerivative}
    \snorm \wtilde_h - \what - \What \snorm_{\Rd,s}
    \,\leq\, C \frac{|s|^3}{\real(s)^3}
    \|\uihat\|_{H^1(\BRtilde\setminus\ol{\BR})} \|h\|_{C^1(\di D)}^2 \,.
  \end{equation}

  To establish the pointwise estimate
  \eqref{eq:EstFrechetDerivativeLDpointwise}, we proceed as in the
  proof of \cite[Pro.~2.3]{KnoNic26}.
  Using the fact that
  $\uthat_h(z) - \uthat(z) = \wtilde_h(z) - \what(z)$ and that
  $\uhat_h'(z) = \What(z)$ (since $h(z)=0$) we find
  using Green's representation theorem (with single- and double layer
  potentials on $\di B_{R-\eps}(0)$ for some $\eps>0$ sufficiently
  small such that $\ol{D}\subs B_{R-\eps}(0)$) and~\eqref{eq:PointwiseEstSLDLLD} that  
  \begin{equation*}
    \begin{split}
      | \uthat_h(z)&-\uthat(z)-\uhat_h'(z) |
      \,=\, | \wtilde_h(z)-\what(z)-\What(z) | \\
      &\,=\, \bigl|
      -(\SL(s)(\di_\nu^+(\wtilde_h-\what-\What))(z)
      + (\DL(s)(\gamma^+(\wtilde_h-\what-\What))(z)
      \bigr| \\
      &\,\leq\, C_{\sigma,\eps} \, |s| \, 
      \frac{\rme^{-\dist(z,\di\BR)\real(s)}}{\dist(z,\di\BR)}
      \Bigl( \|\di_\nu^+(\what_h-\what-\What)\|_{H^{-\frac12}(\di B_{R-\eps}(0))} \\
      &\phantom{\,\leq\, C_{\sigma,\eps} \, |s| \, 
        \frac{\rme^{-\dist(z,\di\BR)\real(s)}}{\dist(z,\di\BR)}
        \Bigl(}
      + |s|^{\frac{1}{2}} \| \gamma^+(\wtilde_h-\what-\What) \|_{H^{\frac12}(\di B_{R-\eps}(0))}
      \Bigr) \,.
    \end{split}
  \end{equation*}
  Applying \cite[Lmm.~4.5]{BanSay22}, the continuity of
  the trace operator
  $\gamma^+: H^1(\Rd\setminus\ol{B_{R-\eps}(0)})
  \to H^{\frac12}(\di B_{R-\eps}(0))$ and \eqref{eq:EquivalenceSnorm} we
  obtain that 
  \begin{equation*}
    \begin{split}
      | \uthat_h(z)-\uthat(z)-\uhat_h'(z) |
      &\,\leq\, C_{\sigma,\eps} 
      \frac{\rme^{-\dist(x,\di\BR)\real(s)}}{\dist(x,\di\BR)}
      \, |s|^{\frac{3}{2}} \, 
      \snorm \what_h-\what-\What\snorm_{\Rd\setminus\ol{D},s} \\
      &\,\leq\, C_{\sigma,\eps} \,
      \frac{\rme^{-\dist(x,\di\BR)\real(s)}}{\dist(x,\di\BR)}
      \, \frac{|s|^{\frac{9}{2}}}{\real(s)^3} \,
      \|\uihat\|_{H^1(\BRtilde\setminus\ol{\BR})} \| h \|_{C^1(\di D)}^2 \,.
    \end{split}
  \end{equation*}
  For the second inequality, we have used~\eqref{eq:EstMaterialDerivative}.
  This shows \eqref{eq:EstFrechetDerivativeLDpointwise}. 
\end{proof}

Using Green's representation theorem, we can write the weak solution
$\uhat_h'\in H^1(D^+)$ 
of~\eqref{eq:TransmissionProblemFDLD} as 
\begin{align*}
  \uhat_h' &\,=\, - \SL(s)(\di_\nu^+\uhat_h') + \DL(s)(\gamma^+\uhat_h') 
  &&\text{in } D^+ \,,\\
  \uhat_h' &\,=\, \SL(s_0)(\di_\nu^-\uhat_h') - \DL(s_0)(\gamma^-\uhat_h') 
  &&\text{in } D^- \,,
\end{align*}
where as before $s_0 := s/c_0$.
In the following we denote by
\begin{subequations}
  \label{eq:DefEtahatprimeZetahatprime}
  \begin{align}
    \etahat'
    &\,:=\, h_\nu \Bigl( 1 - \frac{1}{\rho_0} \Bigr) \di_\nu^-\uthat  \,,
      \label{eq:DefEtahatprime}\\ 
    \zetahat'
    &\,:=\, - s^2  h_\nu \Bigl( 1 - \frac{1}{\rho_0 c_0^2} \Bigr) \gamma^-\uthat
      + \sdiv \Bigl( h_\nu \Bigl( 1-\frac{1}{\rho_0} \Bigr) \sgrad
      \gamma^- \uthat \Bigr) \,, \label{eq:DefEtahatprimeZetahatprime2}
  \end{align}
\end{subequations}
the right hand sides of the transmission conditions
in~\eqref{eq:TransmissionProblemFDLD}.
Applying the jump relations
for the single and double layer potential on~$\di D$, we find that 
$(\frac1s\di_\nu^+\uhat_h',-\gamma^+\uhat_h')\in \HminhalfdiD\times\HhalfdiD$
satisfies the integral equation 
\begin{equation}
  \label{eq:BoundaryIntegralEqn_DomainDerivative}
  \Ccal(s)
  \begin{bmatrix}
    \frac{1}{s}\di_\nu^+\uhat_h' \\
    -\gamma^+\uhat_h'
  \end{bmatrix}
  \,=\,
  \Acal(s)
  \begin{bmatrix}
    -\frac{1}{s}\zetahat' \\
    \etahat'
  \end{bmatrix}
  \qquad \text{on } \di D \,.
\end{equation}
Here, $\Ccal(s)$ and $\Acal(s)$ are the integral operators introduced
in \eqref{eq:DefOpCcal} and \eqref{eq:DefOpAcal}, respectively.
The operator $\Ccal(s)$ has a bounded inverse
$\Ccal^{-1}: \Xcal'\to \Xcal$ for any $s\in\C_+$; see Appendix~\ref{app:A}. 
Accordingly, we obtain the following explicit representation of the
domain derivative $\uhat_h'\in H^1(\Rd\setminus\di D)$ of the
scattered field in the Laplace domain, 
\begin{subequations}
  \label{eq:ExplRepDomainDerivativeLD}
  \begin{align}
    \uhat_h'
    &\,=\, - \begin{bmatrix}
      s\SL(s) & \DL(s)
    \end{bmatrix} \Ccal^{-1}(s) \Acal(s) \begin{bmatrix}
      -\frac{1}{s}\zetahat' \\
      \etahat'
    \end{bmatrix}
    &&\text{in } D^+ \,,\\
    \uhat_h' 
    &\,=\, \begin{bmatrix}
      \rho_0s \SL(s_0) & \DL(s_0) 
    \end{bmatrix}
                         \biggl( \Ccal^{-1}(s) \Acal(s) -
                         \begin{bmatrix}
                           0 & I \\ I & 0
                         \end{bmatrix}
                                        \biggr)
                                        \begin{bmatrix}
                                          -\frac{1}{s}\zetahat' \\
                                          \etahat'
                                        \end{bmatrix}
    &&\text{in } D^- \,.
  \end{align}  
\end{subequations}

\begin{remark}
  \label{rem:GeneralTransmissionFunctions}
  For later reference we note that 
  \eqref{eq:DefEtahatprimeZetahatprime}--\eqref{eq:ExplRepDomainDerivativeLD} 
  and therefore also the transmission 
  problem~\eqref{eq:TransmissionProblemFDLD} still make sense when we 
  replace~$(\uthat,\ushat) \in H^1(D^-)\times H^1(D^+)$ by the solution 
  $(\vhat^-,\vhat^+)\in H^1(D^-)\times H^1(D^+)$ of the transmission
  problem~\eqref{eq:TransmissionProblemLD} with arbitrary transmission
  functions $\etahat \in \HthreehalfdiD$ and $\zetahat \in \HhalfdiD$. 
  In this case, one has to use $\di_\nu^-\vhat^-$ and $\gamma^-\vhat^-$ instead 
  of $\di_\nu^-\uthat$ and $\gamma^-\uthat$, respectively, 
  in~\eqref{eq:DefEtahatprimeZetahatprime} for the right hand side of 
  the transmission conditions in~\eqref{eq:TransmissionProblemFDLD}.
  This generalization will be utilized in Proposition~\ref{pro:EstimatesU_h'LD} 
  and in Theorem~\ref{thm:CharacterizationTemporalDomainDerivativeTD} in
  order to obtain the continuous extensions in \eqref{eq:ContinuousExtensionG} 
  and \eqref{eq:ContinuousExtensionGinfty}.~\hfill$\lozenge$
\end{remark}

Next we discuss the domain derivative of the shifted far field
pattern $\uinfty_R$ from \eqref{eq:ShiftedFarfieldPatternLD}. 
Similar to \eqref{eq:OperatorFhatD}, we define for any $\xi\in\Sd$ the 
non-linear operator
\begin{equation}
  \label{eq:OperatorFhatDinfty}
  \Fhat_{D,R}^\infty :\;
  \Dcal(\Fhat_{D,R}^\infty) \subset C^1(\di D,\Rd) \to \C \,,
  \qquad \Fhat_{D,R}^\infty(h) \,:=\, \uinftyhat_{R,h}(\xi) \,,
\end{equation}
where $\uinftyhat_{R,h} \in L^2(\Sd)$ denotes the shifted far field
pattern of the solution of the transmission
problem~\eqref{eq:TransmissionProblemLD} with~$D$ replaced by $D_h$
and with~$\etahat = - \gamma\uihat|_{\di D_h}$ and 
  $\zetahat = -\di_\nu\uihat|_{\di D_h}$, as defined 
in~\eqref{eq:ShiftedFarfieldPatternLD}. 
Here, $\Dcal(\Fhat_{D,R}^\infty)$ coincides with $\Dcal(\Fhat_D)$ 
from \eqref{eq:DefDomainFhatD}.

\begin{corollary}
  \label{cor:FrechetDerivativeLD}
  The operator $\Fhat_{D,R}^\infty$ is Fr\'echet differentiable at 
  zero.
  The Fr\'echet derivative is given
  by~$(\Fhat_{D,R}^\infty)'(0)h = (\uhat_h')_R^\infty(\xi)$, where
  $(\uhat_h')_R^\infty \in L^2(\Sd)$ denotes the shifted far field 
  pattern 
  \begin{equation}
    \label{eq:DomainDerivativeFarfield}
    (\uhat_h')_R^\infty(\xi,s)
    \,:=\, \frac{1}{4\pi} \rme^{-sR} \int_{\di D}
    \rme^{s \xi\cdot y} 
    \bigl(
    (\nu_y\cdot\xi) s(\gamma^+\uhat_h')(y,s) 
    - (\di_\nu^+ \uhat_h')(y,s) \bigr) \ds_y \,,
  \end{equation}
  of the domain derivative in the Laplace domain $\uhat_h'$ from
  Proposition~\ref{pro:FrechetDerivativeLD}.
  The function $(\uhat_h')_R^\infty$ is called the domain derivative
  of the shifted far field pattern in the Laplace domain.
  We have
  \begin{equation}
    \label{eq:EstFrechetDerivativeFarfieldLDpointwise}
    \bigl| \Fhat_{D,R}^\infty(h) - \Fhat_{D,R}^\infty(0)
    - (\Fhat_{D,R}^\infty)'(0)h \bigr|
    \,\leq\, C_\sigma \rme^{-\eps \real(s)}
    \frac{|s|^4}{\real(s)^3} \|\uihat\|_{H^1(\BRtilde\setminus\ol{\BR})}
    \|h\|_{C^1(\di D)}^2     
  \end{equation}
  for all $s\in\C_+$ with $\real(s)\geq \sigma>0$ and for any $\Rtilde>R$. 
\end{corollary}

\begin{proof}
  Using \eqref{eq:ShiftedFarfieldPatternLD} and
  \eqref{eq:DomainDerivativeFarfield} we find that
  \begin{equation*}
    \begin{split}
      \bigl|& \Fhat_{D,R}^\infty(h) - \Fhat_{D,R}^\infty(0)
      - (\Fhat_{D,R}^\infty)'(0)h \bigr|
      \,=\, 
      \bigl| \uinftyhat_{h,R}(\xi) - \uinftyhat_R(\xi)
      - (\uhat_h')_R^\infty(\xi) \bigr| \\
      &\,=\, \biggl| \frac{1}{4\pi} \rme^{-sR} \int_{\di B_{R-\eps}(0)}
      \rme^{s \xi\cdot y} 
      \Bigl( (\nu_y\cdot\xi) s
      \bigl(\gamma (\uhat_h-\uhat-\uhat_h')\bigr)(y) 
      - \bigl(\di_{\nu_y} (\uhat_h-\uhat-\uhat_h')\bigr)(y) \Bigr) \ds_y
      \biggr| \,.
    \end{split}
  \end{equation*}
  In the second step we used Green's formula to transfer the integral 
  representations of the far field patterns $\uinftyhat_{h,R}$, 
  $\uinftyhat_R$, and $(\uhat_h')_R^\infty$ from $\di D_h$ and $\di D$, 
  respectively, to~$\di B_{R-\eps}(0)$. 
  Next we use again the notation from the proof of 
  Proposition~\ref{pro:FrechetDerivativeLD}.
  Since $\gamma (\uthat_h-\uthat-\uhat_h')=\gamma (\wtilde_h-\what-\What)$ on
  $\di B_{R-\eps}(0)$, we obtain that
  \begin{equation*}
    \begin{split}
      \bigl|& \uinftyhat_{h,R}(\xi) - \uinftyhat_R(\xi)
      - (\uhat_h')_R^\infty(\xi) \bigr| \\
      &\,\leq\, C \rme^{-\eps\real(s)} \Bigl( 
      |s|
      \bigl\|\gamma(\wtilde_h-\what-\What)\bigr\|_{H^{\frac12}(\di B_{R-\eps})} 
      + \bigl\| \di_{\nu_y}
      (\wtilde_h-\what-\What)\bigr\|_{H^{-\frac12}(\di B_{R-\eps})}
      \Bigr) \,.
    \end{split}
  \end{equation*}
  Applying \cite[Lmm.~4.5]{BanSay22}, the continuity of
  the trace operator $\gamma^+$, \eqref{eq:EquivalenceSnorm} as well
  as ~\eqref{eq:EstMaterialDerivative}, we obtain similar to the last
  step in the proof of Theorem~\ref{pro:FrechetDerivativeLD} that 
  \begin{equation*}
    \begin{split}
      \bigl| \uinftyhat_{h,R}(\xi) - \uinftyhat_R(\xi)
      - (\uhat_h')_R^\infty(\xi) \bigr|
      \,\leq\, C_\sigma \rme^{-\eps\real(s)}
      \frac{|s|^4}{\real(s)^3}
      \|\uihat\|_{H^1(\BRtilde\setminus\ol{\BR})} \|h\|_{C^1(\di D)}^2 \,.
    \end{split}
  \end{equation*}
  This shows \eqref{eq:EstFrechetDerivativeFarfieldLDpointwise}. 
\end{proof}

From \eqref{eq:ExplRepDomainDerivativeLD} we immediately obtain the
following explicit representation of the shifted far field pattern of
the domain derivative, 
\begin{equation}
  \label{eq:ExplRepDomainDerivativeFarfieldLD}
  (\uhat_h')^\infty_R
  \,=\, - \begin{bmatrix}
    s\SL^\infty_R(s) & \DL^\infty_R(s)
  \end{bmatrix} \Ccal^{-1}(s) \Acal(s) \begin{bmatrix}
    -\frac{1}{s}\zetahat' \\
    \etahat'
  \end{bmatrix}
  \qquad \text{on } \Sd \,,
\end{equation}
where $\etahat'$ and $\zetahat'$ are as 
in~\eqref{eq:DefEtahatprimeZetahatprime} with $\gamma^-\uthat$ 
and $\di_\nu^-\uthat$ denoting the trace and the normal 
derivative of the solution to the transmission 
problem~\eqref{eq:TransmissionProblemLD} with 
$\etahat = - \gamma\uihat$ and $\zetahat = -\di_\nu\uihat$.
This representation and the associated integral
equation will be used for the numerical implementation of the temporal
domain derivative in Section~\ref{sec:NumericalExamples} below. 

In the next proposition we establish bounds on the domain
derivative $\uhat_h'$ and on its shifted far field pattern
$(\uhat_h')^\infty_R$ that are explicit with respect to the
parameter~$s$.
These will be required to characterize the temporal domain derivative
in Theorem~\ref{thm:CharacterizationTemporalDomainDerivativeTD} below. 
We consider the case of general transmission functions as explained in 
Remark~\ref{rem:GeneralTransmissionFunctions}.

\begin{proposition}
  \label{pro:EstimatesU_h'LD}
  Suppose $\real(s)\geq\sigma>0$, $h\in C^1(\di D,\Rd)$, 
  and let $(\vhat^-,\vhat^+) \in H^1(D^-)\times H^1(D^+)$ be the 
  solution of the transmission problem~\eqref{eq:TransmissionProblemLD} 
  with arbitrary transmission functions $\etahat \in \HthreehalfdiD$ 
  and~$\zetahat \in \HhalfdiD$.
  We consider the associated solution $\uhat_h'\in H^1(D^-)\times H^1(D^+)$ 
  of the transmission problem~\eqref{eq:TransmissionProblemFDLD}, where 
  the right hand sides of the transmission conditions are given 
  by~\eqref{eq:DefEtahatprimeZetahatprime} with~$\di_\nu^-\vhat^-$ 
  and $\gamma^-\vhat^-$ instead of $\di_\nu^-\uthat$ and 
  $\gamma^-\uthat$, respectively.
  Then,
  \begin{equation}
    \label{eq:BoundUhat'H1LD}
    \|\uhat_h'\|_{H^1(D^\pm)}
    \,\leq\, C_\sigma \|h\|_{C^1(\di D)}
    \frac{|s|^5}{\real(s)^{2}}
    \Bigl( \| \zetahat \|_{\HhalfdiD}^2 
    + \| \etahat \|_{H^{\frac32}(\di D)}^2 \Bigr)^{\frac12} \,,
    \qquad \real(s)\geq\sigma>0 \,,
  \end{equation}
  and
  \begin{equation}
    \label{eq:BoundUhat'L2LD}
    \|\uhat_h'\|_{L^2(D^\pm)}
    \,\leq\, C_\sigma
    \|h\|_{C^1(\di D)}
    \frac{|s|^4}{\real(s)^{2}}
    \Bigl( \| \zetahat \|_{\HhalfdiD}^2 
    + \| \etahat \|_{H^{\frac32}(\di D)}^2 \Bigr)^{\frac12} \,,
    \qquad \real(s)\geq\sigma>0 \,.
  \end{equation}
  Furthermore, for any $z\in\Rd\setminus\ol{\BR}$ and 
  $\real(s)\geq\sigma>0$, we have the pointwise estimate 
  \begin{equation}
    \label{eq:BoundUhat'PointwiseLD}
    |\uhat_h'(z)|
    \,\leq\, C_\sigma \|h\|_{C^1(\di D)}
    \frac{\rme^{-\dist(x,\di D)\real(s)}}{\dist(x,\di D)}
    \frac{|s|^\frac{13}{2}}{\real(s)^{2}}
    \Bigl( \| \zetahat \|_{\HhalfdiD}^2 
    + \| \etahat \|_{H^{\frac32}(\di D)}^2 \Bigr)^{\frac12} \,.
  \end{equation}
  The shifted far field pattern of the domain derivative satisfies,
  for any $\xi\in\Sd$ and $\real(s)\geq\sigma>0$, 
  \begin{equation}
    \label{eq:BoundUinftyhat'PointwiseLD}
    \bigl| (\uhat_h')_R^\infty(\xi) 
    \bigr| \,\leq\,
    C_\sigma \|h\|_{C^1(\di D)} \rme^{-\real(s)(R-\delta_D(\xi))}
    \frac{|s|^6}{\real(s)^{2}}
    \Bigl( \| \zetahat \|_{\HhalfdiD}^2 
    + \| \etahat \|_{H^{\frac32}(\di D)}^2 \Bigr)^{\frac12} \,.
  \end{equation}
\end{proposition}

\begin{proof}
  We use again the notations $\etahat'$ and $\zetahat'$ for the right hand
  sides of the transmission conditions in~\eqref{eq:TransmissionProblemFDLD}. 
  Proceeding as in the proof of Lemma~\ref{lmm:StabilityUshat} we obtain that
  \begin{equation}
    \label{eq:ProofEstDDLD2}
    \snorm \uhat_h' \snorm_{D^-,s} + \snorm \uhat_h' \snorm_{D^+,s}
    \,\leq\, C_\sigma \frac{|s|^{\frac32}}{\real(s)}
    \Bigl( \| \zetahat' \|_{\HminhalfdiD} + \|\etahat'\|_{\HhalfdiD} \Bigr) \,.
  \end{equation}

  Using \eqref{eq:DefEtahatprimeZetahatprime} and the continuity of
  $\di_\nu^-:H^2(D)\to\HhalfdiD$, we find that
  \begin{equation*}
      \|\etahat'\|_{\HhalfdiD}
      = \Bigl\|  h_\nu \Bigl(1- \frac{1}{\rho_0} \Bigr) \di_\nu^-\vhat^- \Bigr\|_{\HhalfdiD} 
      \leq C \|h\|_{C^1(\di D)} \| \di_\nu^-\vhat^- \|_{\HhalfdiD} 
      \leq C \|h\|_{C^1(\di D)} \| \vhat^- \|_{H^2(D)} \,. 
  \end{equation*}
  Applying \eqref{eq:ScatteringProblemLD_H2estimate} gives
  \begin{equation}
    \label{eq:ProofEstDDLD3}
    \|\etahat'\|_{\HhalfdiD}
    \,\leq\, C_\sigma \|h\|_{C^1(\di D)} 
    \frac{|s|^\frac{5}{2}}{\real(s)^{\frac12}}
    \Bigl( \|\zetahat\|_{\HhalfdiD}^2
    + \|\etahat\|_{\HthreehalfdiD}^2 \Bigr)^{\frac12} 
  \end{equation}
  for $s\in\C_+$ with $\real(s)\geq\sigma>0$.
  Similarly,
  \begin{equation*}
    \begin{split}
      \| \zetahat' \|_{\HminhalfdiD}
      &\,=\, \Bigl\| s^2 h_\nu
      \Bigl( 1 - \frac{1}{\rho_0 c_0^2} \Bigr) \gamma\vhat^-
      + \sdiv \Bigl( h_\nu \Bigl( 1-\frac{1}{\rho_0} \Bigr) \sgrad
      \gamma \vhat^- \Bigr) \Bigr\|_{\HminhalfdiD} \\
      &\,\leq\, C |s|^2 \|h\|_{C^1(\di D)} \|\gamma\vhat^-\|_{\HhalfdiD}
      + C \|h\|_{C^1(\di D)} \|\gamma\vhat^-\|_{H^{\frac32}(\di D)} \\
      &\,\leq\, C  \|h\|_{C^1(\di D)} \bigl(
      |s|^2 \|\vhat^-\|_{H^1(D^-)} + \|\vhat^-\|_{H^2(D^-)} \bigr) \,,
    \end{split}
  \end{equation*}
  and \eqref{eq:EquivalenceSnorm}, \eqref{eq:StabilityUshat},
  and~\eqref{eq:ScatteringProblemLD_H2estimate} give 
  \begin{equation}
    \label{eq:ProofEstDDLD4}
    \begin{split}
      \| \zetahat' \|_{\HminhalfdiD}
      &\,\leq\, C_\sigma  \|h\|_{C^1(\di D)}
      \frac{|s|^\frac{7}{2}}{\real(s)}
      \Bigl( \|\zetahat\|_{\HhalfdiD}^2
      + \|\etahat\|_{\HthreehalfdiD}^2 \Bigr)^{\frac12} \,. 
    \end{split}
  \end{equation}
  Substituting \eqref{eq:ProofEstDDLD3} and \eqref{eq:ProofEstDDLD4}
  into \eqref{eq:ProofEstDDLD2} yields
  \begin{equation}
    \label{eq:ProofEstDDLD5}
    \snorm \uhat_h' \snorm_{D^-,s} + \snorm \uhat_h' \snorm_{D^+,s}
    \,\leq\, C_\sigma \|h\|_{C^1(\di D)}
    \frac{|s|^5}{\real(s)^2}
    \Bigl( \| \zetahat \|_{\HhalfdiD}^2 
    + \| \etahat \|_{H^{\frac32}(\di D)}^2 \Bigr)^{\frac12} \,,
  \end{equation}
  which implies~\eqref{eq:BoundUhat'H1LD}.
  The inequality~\eqref{eq:BoundUhat'L2LD} also follows immediately
  from~\eqref{eq:ProofEstDDLD5} by dropping the gradient term in
  $\snorm \uhat_h' \snorm_{D^\pm,s}$ and dividing the
  resulting inequality by $|s|$. 

  To show \eqref{eq:BoundUhat'PointwiseLD} we observe that
    \cite[Lmm.~4.5]{BanSay22} gives
    \begin{equation}
      \label{eq:ProofEstDDLD1a}
      \begin{split}
        \|\di_\nu^+\uhat_h'\|_{\HminhalfdiD} 
        &\,\leq\, C \max(1,|s|^{-\frac12}) |s|^{\frac12} 
        \snorm \uhat_h' \snorm_{D^+,s} \,.
      \end{split}
    \end{equation}
  We combine
  \eqref{eq:ProofEstDDLD1a} and \eqref{eq:ProofEstDDLD5} to see that
  \begin{equation*}
    \|\di_\nu^+\uhat_h'\|_{\HminhalfdiD}
    \,\leq\, C_\sigma \|h\|_{C^1(\di D)}
    \frac{|s|^{\frac{11}{2}}}{\real(s)^2}
    \Bigl( \| \zetahat \|_{\HhalfdiD}^2 
    + \| \etahat \|_{H^{\frac32}(\di D)}^2 \Bigr)^{\frac12} \,. 
  \end{equation*}
  Moreover, using the continuity of the trace operator
  $\gamma^-: H^1(D) \to \HhalfdiD$ and \eqref{eq:EquivalenceSnorm}, 
  and combining the result with \eqref{eq:ProofEstDDLD5}, we obtain 
  \begin{equation*}
    \begin{split}
      \|\gamma^+\uhat_h'\|_{\HhalfdiD}
      &\,\leq\, C \|\uhat_h'\|_{H^1(D^+)}
      \,\leq\, C \max(1,|s|^{-1}) \snorm \uhat_h' \snorm_{D^+,s} \\
      &\,\leq\, C_\sigma \|h\|_{C^1(\di D)}
      \frac{|s|^5}{\real(s)^2}
      \Bigl( \| \zetahat \|_{\HhalfdiD}^2 
      + \| \etahat \|_{H^{\frac32}(\di D)}^2 \Bigr)^{\frac12} \,.
    \end{split}
  \end{equation*}
  Using Green's representation theorem and \eqref{eq:PointwiseEstSLDLLD},
  we finally arrive at
  \begin{equation*}
    \begin{split}
      |\uhat_h'(z)|
      &\,\leq\, \bigl| -\bigl(\SL(s)(\di_\nu^+\uhat_h')\bigr)(z)
      + \bigl(\DL(s)(\gamma^+\uhat_h')\bigr)(z) \bigr| \\
      &\,\leq\,  C_\sigma
      \frac{\rme^{-\dist(x,\di D)\real(s)}}{\dist(x,\di D)}
      \frac{|s|^{\frac{13}{2}}}{\real(s)^2}
      \Bigl( \| \zetahat \|_{\HhalfdiD}^2 
      + \| \etahat \|_{H^{\frac32}(\di D)}^2 \Bigr)^{\frac12} \,. 
    \end{split}
  \end{equation*}

  To see \eqref{eq:BoundUinftyhat'PointwiseLD}, we use
  \eqref{eq:DomainDerivativeFarfield}, \cite[Lmm.~4.5]{BanSay22}, the
  continuity of the trace operator $\gamma^+$,
  \eqref{eq:EquivalenceSnorm}, and~\eqref{eq:ProofEstDDLD5} to estimate 
  \begin{equation*}
    \begin{split}
      \bigl| (\uhat_h')_R^\infty(\xi) \bigr|
      &\,\leq\, C \rme^{-\real(s)(R-\delta_D(\xi))} 
      \Bigl(
      |s|\|\gamma^+\uhat_h'\|_{\HhalfdiD}
      + \|\di_{\nu}^+ \uhat_h'\|_{\HminhalfdiD} \Bigr) \\
      &\,\leq\, C_\sigma \|h\|_{C^1(\di D)}
      \rme^{-\real(s)(R-\delta_D(\xi))}
      \frac{|s|^6}{\real(s)^2}
      \Bigl( \| \zetahat \|_{\HhalfdiD}^2 
      + \| \etahat \|_{H^{\frac32}(\di D)}^2 \Bigr)^{\frac12} \,.
    \end{split}
  \end{equation*}
\end{proof}

\section{Scattering problem and domain derivative in the time domain}
\label{sec:TimeDomain}
To translate the results obtained in
Sections~\ref{sec:ScatteringProblemLD}
and~\ref{sec:DomainDerivativeLD} from the Laplace domain into the time
domain we recall some basic facts about analytic families of linear
operators and their Laplace transform; see, e.g.,
\cite{BanSay22,Lub94,Tre75} for further details.  
Throughout this section, we fix a final observation time~$T>0$.

Suppose that $K(s):X\to Y$, $\real(s)\geq \sigma > 0$, is an analytic
family of bounded linear operators between two Hilbert spaces $X$ and
$Y$, which is bounded by 
\begin{equation*}
  \| K(s) \|
  \,\leq\, C_\sigma \frac{|s|^\mu}{\real(s)^\nu} \,, \qquad
  \real(s) \geq \sigma > 0 \,,
\end{equation*}
for some $\mu\in\R$ and~$\nu\geq 0$ with respect to the operator
norm. 
Then $K(s)$ is the Laplace transform of a bounded linear convolution
operator $K(\di_t): H^{r+\mu}_0(0,T;X) \to H^r_0(0,T;Y)$ for arbitrary
$r\in\R$, with
\begin{equation*}
  \Lcal\bigl( K(\di_t)g \bigr)(s)
  \,=\, K(s) (\Lcal g)(s) \,, \qquad \real(s) \geq \sigma > 0 \,.
\end{equation*}
Here we use the usual Heaviside operational notation for the
convolution operator in the time domain; see,
e.g.,~\cite{BanSay22,Lub94}.
Moreover, the composition rule
\begin{equation*}
  K_1(\di_t) K_2(\di_t) g
  \,=\, (K_1 K_2)(\di_t) g
\end{equation*}
holds for any two such families of operators $K_1(s):Y \to Z$
and $K_2(s):X \to Y$ with Hilbert spaces~$X, Y$, and $Z$.

\subsection{Scattering problem in the time domain}
\label{subsec:ScatteringProblemTimeDomain}
From the integral representations and estimates developed in
Section~\ref{sec:ScatteringProblemLD} we immediately infer the
following representations and well-posedness results in the time
domain. 

\begin{theorem}
  \label{thm:WellPosednessScatterinProblemTD}
  Let $r\in\R$, and let $\ui$ be an incident 
  plane wave as in~\eqref{eq:PlaneWave} with 
  $f^i \in H^{r+5}_c(\R)$, which satisfies~\eqref{eq:WaveEquationUi}.
  Accordingly, we obtain
  \begin{equation*}
   (\di_\nu\ui,\gamma\ui) 
   \in H^{r+3}_0\bigl(0,T;\HhalfdiD\times\HthreehalfdiD\bigr) 
   \cap H^{r+4}_0\bigl(0,T;\HminhalfdiD\times\HhalfdiD\bigr) \,.   
  \end{equation*}
  The unique solution of the time-dependent scattering
  problem~\eqref{eq:TransmissionProblemTD} with homogeneous initial
  conditions~$\us(\ph,0) = \dit\us(\ph,0) = 0$ in $D^+$ and
  $\ut(\ph,0) = \dit\ut(\ph,0) = 0$ in $D^-$ is given by 
  \begin{subequations}
    \label{eq:ExplRepSolnTransmissionProblemTD}
  \begin{align}
    \us
    &\,=\, - \begin{bmatrix}
      \di_t\SL(\di_t) & \DL(\di_t)
    \end{bmatrix} \Ccal^{-1}(\di_t) \Acal(\di_t) \begin{bmatrix}
      \di_t^{-1}\di_\nu\ui \\
      -\gamma\ui 
    \end{bmatrix} 
    &&\text{in } D^+ \times (0,\infty) \,, \label{eq:usinDp}\\
    \ut 
    &\,=\, \begin{bmatrix}
      \rho_0 \di_t \SL(c_0\di_t) & \DL(c_0\di_t) 
    \end{bmatrix} \bigg(
                                   \begin{bmatrix}
                                     I & 0 \\ 0 & I
                                   \end{bmatrix}
                                                  + \Ccal^{-1}(\di_t) \Acal(\di_t) \bigg)
                                                  \begin{bmatrix}
                                                    \di_t^{-1} \di_\nu\ui \\
                                                    -\gamma\ui 
                                                  \end{bmatrix} 
    &&\text{in } D^- \times (0,\infty) \,.
  \end{align}
  \end{subequations}
  The operator
  $\Fcal: ( \di_\nu\ui \,,\, \gamma\ui  )
  \mapsto  ( \ut \,,\, \us )$ 
  has continuous extensions
  \begin{subequations}
    \label{eq:ContinuousExtensionF}
    \begin{align}
      \Fcal :\, H^{r+4}_0\bigl(0,T;\HminhalfdiD\times\HhalfdiD\bigr)
      \to H^{r+\frac52}_0\bigl(0,T;H^1(D^-)\times H^1(D^+)\bigr) \,,\\
      \Fcal :\, H^{r+3}_0\bigl(0,T;\HhalfdiD\times\HthreehalfdiD\bigr)
      \to H^{r+\frac12}_0\bigl(0,T;H^2(D^-)\times H^2(D^+)\bigr) \,.
    \end{align}
  \end{subequations}
  For any fixed $z \in \Rd\setminus\ol{\BR}$, the operator 
  $\Fcal_z: ( \di_\nu\ui \,,\, \gamma\ui )
  \mapsto  \us(z)$ has a continuous extension
  \begin{equation}
    \label{eq:ContinuousExtensionFz}
    \Fcal_z :\, H^{r+4}_0\bigl(0,T;\HminhalfdiD\times\HhalfdiD\bigr)
    \to H^{r+1}_0\bigl(0,T;\R\bigr) \,.
  \end{equation}
  Moreover, the shifted far field pattern in the time-domain 
  from~\eqref{eq:ShiftedFarfieldPatternTD} can be written as
  \begin{equation}
    \uinfty_R
    \,=\,- \begin{bmatrix}
      \di_t\SL^\infty_R(\di_t) & \DL^\infty_R(\di_t)
    \end{bmatrix} \Ccal^{-1}(\di_t) \Acal(\di_t) \begin{bmatrix}
      \di_t^{-1}\di_\nu\uihat \\
      -\gamma\uihat 
    \end{bmatrix} 
    \qquad \text{on } \Sd \times (0,\infty) \,, \label{eq:uRinfty}
  \end{equation}
  and for each $r\in\R$ the operator
  $\Fcal^\infty_\xi: ( \di_\nu\ui \,,\, \gamma\ui  )
  \mapsto  \uinfty_R(\xi,\ph)$ has a continuous extension
  \begin{equation}
    \label{eq:ContinuousExtensionFinfty}
    \Fcal^\infty_\xi:\, H^{r+4}_0\bigl(0,T;\HminhalfdiD\times\HhalfdiD\bigr)
    \to H^{r+\frac32}_0\bigl(0,T;\R\bigr) \,. 
  \end{equation}
\end{theorem}

\begin{proof}
  This follows from the corresponding explicit representations in the
  Laplace domain in~\eqref{eq:ExplRepSolnTransmissionProblemLD}
  and~\eqref{eq:ExplRepFarfieldTransmissionProblemLD}, and from the
  estimates in Lemma~\ref{lmm:StabilityUshat}, Corollary~\ref{cor:LDpwbound}
  and the Propositions~\ref{pro:ScatteringProblemLD_H2estimate}
  and~\ref{pro:EstFarfieldPatternPointwiseLD}.
  For the continuous extensions in \eqref{eq:ContinuousExtensionF},
  \eqref{eq:ContinuousExtensionFz}, and 
  \eqref{eq:ContinuousExtensionFinfty} we use that
  the estimates \eqref{eq:StabilityUshat}, \eqref{eq:vhatboundpwLD},
  \eqref{eq:ScatteringProblemLD_H2estimate},
  and~\eqref{eq:EstFarfieldPatternPointwiseLD} hold for arbitrary 
  transmission functions $\etahat$ and $\zetahat$.
\end{proof}

\subsection{Domain derivative in the time domain}

Next we consider the well-posedness of the time-domain analog of the
transmission problem \eqref{eq:TransmissionProblemFDLD}, which is then
used to characterize the temporal domain derivative in the
Theorems~\ref{thm:TemporalDomainDerivativeNearfield} and
\ref{thm:TemporalDomainDerivativeFarfield} below.
Let $r\in\R$, and let $\ui$ be an incident plane wave as in 
Theorem~\ref{thm:WellPosednessScatterinProblemTD}.
Then we denote by~$(\ut, \us)
\in H^{r+\frac12}_0\bigl(0,T; H^2(D^-) \times H^2(D^+)\bigr)$ the solution 
of the time-dependent scattering problem~\eqref{eq:TransmissionProblemTD}
with initial conditions
$\us(\ph,0) = \dit\us(\ph,0) = 0$ in $D^+$ and
$\ut(\ph,0) = \dit\ut(\ph,0) = 0$ in~$D^-$ as in
Theorem~\ref{thm:WellPosednessScatterinProblemTD}. 
Accordingly, we use the notation
\begin{subequations}
  \label{eq:DefEtaZetaTD}
  \begin{align}
    \eta'
    &\,:=\, h_\nu \Bigl( 1 - \frac{1}{\rho_0} \Bigr) \di_\nu^-\ut \,, \\
    \zeta'
    &\,:=\, - h_\nu \Bigl( 1 - \frac{1}{\rho_0 c_0^2} \Bigr)
      \di_t^2 \gamma^{-} \ut
      + \sdiv \Bigl( h_\nu \Bigl( 1-\frac{1}{\rho_0} \Bigr) \sgrad
      \gamma^{-} \ut \Bigr) \,.
  \end{align}
\end{subequations}

\begin{theorem}
  \label{thm:CharacterizationTemporalDomainDerivativeTD}
  The unique solution
  $u_h' \in H^{r-2}_0(0,T;H^1(\Rd\setminus\di D))$
  of the time-dependent transmission problem
  \begin{subequations}
    \label{eq:TransmissionProblemFDTD}
    \begin{align}
      \di_t^2 u_h' - \Delta u_h'
      &\,=\, 0  
      &&\text{in } D^+\times(0,\infty) \,, \\
      \frac{1}{c_0^2} \di_t^2u_h' - \Delta u_h' 
      &\,=\, 0  
      &&\text{in } D^-\times(0,\infty) \,, \\
      \gamma^+u_h' - \gamma^-u_h'
      &\,=\, \eta' 
      &&\text{on } \di D\times(0,\infty) \,, \\
      \di_\nu^+u_h'
      - \frac{1}{\rho_0} \di_\nu^-u_h'
      &\,=\, \zeta' 
      &&\text{on } \di D\times(0,\infty) \,,
    \end{align}
  \end{subequations}
  with homogeneous initial conditions
  $u_h'(\ph,0) = \dit u_h'(\ph,0) = 0$ in $\Rd$ 
  is given by
  \begin{subequations}
  \begin{align}
    u_h'|_{D^+}
    &\,=\, - \begin{bmatrix}
      \di_t\SL(\di_t) & \DL(\di_t)
    \end{bmatrix} \Ccal^{-1}(\di_t) \Acal(\di_t) \begin{bmatrix}
      -\di_t^{-1}\zeta' \\
      \eta'
    \end{bmatrix}
    &&\text{in } D^+ \times (0,\infty) \,, \label{eq:uprimeinR3}\\
    u_h'|_{D^-} 
    &\,=\, \begin{bmatrix}
      \rho_0 \di_t \SL(c_0 \di_t) & \DL(c_0 \di_t) 
    \end{bmatrix}
                                    \biggl( \Ccal^{-1}(\di_t) \Acal(\di_t) -
                                    \begin{bmatrix}
                                      0 & I \\ I & 0
                                    \end{bmatrix}
                                                   \biggr)
                                                   \begin{bmatrix}
                                                     -\di_t^{-1}\zeta' \\
                                                     \eta'
                                                   \end{bmatrix}
    &&\text{in } D^- \times (0,\infty) \,.
  \end{align}
  \end{subequations}
  Moreover, for any observation point
  $z\in\Rd\setminus\ol{\BR}$ the operator 
  $\Gcal_z: ( \di_\nu\ui \,,\, \gamma\ui ) \mapsto  u_h'(z,\ph)$ 
  has a continuous extension
  \begin{equation}
    \label{eq:ContinuousExtensionG}
    \Gcal_z :\, H^{r+3}_0\bigl(0,T;\HhalfdiD\times\HthreehalfdiD\bigr)
    \to H^{r-\frac72}_0(0,T;\R) \,.
  \end{equation}
  Similarly, the shifted far field pattern of $u_h'$ as defined 
  in~\eqref{eq:ShiftedFarfieldPatternTD} can be written as 
  \begin{equation}
    (u_h')^\infty_R
    \,=\,- \begin{bmatrix}
      \di_t\SL^\infty_R(\di_t) & \DL^\infty_R(\di_t)
    \end{bmatrix} \Ccal^{-1}(\di_t) \Acal(\di_t) 
    \begin{bmatrix}
      -\di_t^{-1}\zeta' \\
      \eta'
    \end{bmatrix} \qquad \text{on } \Sd \times (0,\infty) \,, \label{eq:uRprimeinfty}
  \end{equation}
  and for any observation direction $\xi\in\Sd$ the operator
  $\Gcal^\infty_\xi: ( \di_\nu\ui \,,\, \gamma\ui )
  \mapsto (u_h')^\infty_R(\xi,\ph)$ has a continuous extension
  \begin{equation}
    \label{eq:ContinuousExtensionGinfty}
    \Gcal^\infty_\xi:\, H^{r+3}_0\bigl(0,T;\HhalfdiD\times\HthreehalfdiD\bigr)
    \to H^{r-3}_0(0,T;\R) \,.
  \end{equation}
\end{theorem}

\begin{proof}
  This follows from the corresponding explicit representations in the
  Laplace domain in~\eqref{eq:ExplRepDomainDerivativeLD} and
  \eqref{eq:ExplRepDomainDerivativeFarfieldLD} together with the
  estimates in Proposition~\ref{pro:EstimatesU_h'LD}. 
\end{proof}

Next we consider the time-domain analog of the operator $\Fhat_D$ from
\eqref{eq:OperatorFhatD}, which is given by
\begin{equation}
  \label{eq:OperatorFD}
  F_D:\, \Dcal(F_D) \subs C^1(\di D,\Rd) \to H^{r+1}_0(0,T;\R) \,,
  \qquad F_D(h) \,:=\, \us_h(z,\ph) \,,
\end{equation}
where $(\ut_h,\us_h)$ denotes the solution of the time-dependent 
transmission problem \eqref{eq:TransmissionProblemTD} with~$D$ replaced 
by $D_h$ and with~$\eta = - \gamma\ui|_{\di D_h}$ and 
  $\zeta = -\di_\nu\ui|_{\di D_h}$.
As before, $z\in\Rd\setminus\ol{\BR}$ denotes an arbitrary observation
point and $\Dcal(F_D)$ is the same as $\Dcal(\Fhat_D)$ in \eqref{eq:DefDomainFhatD}.

\begin{theorem}
  \label{thm:TemporalDomainDerivativeNearfield}
  The operator $F_D$ from \eqref{eq:OperatorFD} is Fr\'echet
  differentiable at zero.
  Its Fr\'echet derivative is given by $F_D'(0)h = u_h'(z,\ph)$ for all
  $h\in C^1(\di D,\Rd)$, where
  $u_h' \in H^{r-2}_0(0,T;H^1(\Rd\setminus\di D))$
  is the unique solution of the transmission
  problem~\eqref{eq:TransmissionProblemFDTD} with $\eta'$ and $\zeta'$
  from \eqref{eq:DefEtaZetaTD}. 
  The function $u'_h$ is called the temporal domain derivative of the 
  scattered wave.
  We have
  \begin{equation}
    \label{eq:EstFrechetDerivativeTDpointwise}
    \bigl\| F_D(h) - F_D(0) - F_D'(0)h \bigr\|_{H^{r-\frac12}_0(0,T;\R)}
    \,\leq\, C_T \|\ui\|_{H^{r+4}_0\bigl(0,T;H^1(\BRtilde\setminus\ol{\BR})\bigr)}
    \|h\|_{C^1(\di D)}^2 
  \end{equation}
  for any $\Rtilde>R$. 
\end{theorem}

\begin{proof}
 Combining Plancherel's theorem and
 \eqref{eq:EstFrechetDerivativeLDpointwise}, we find that, for~$\sigma=1/T$, 
 \begin{equation*}
   \begin{split}
     \bigl\| F_D(h) - F_D(0) - F_D'(0)h \bigr\|_{H^{r-\frac12}_0(0,T;\R)}^2
     &\,\leq\, \rme^{2\sigma T} \int_{\sigma+\rmi\R} |s|^{2r-1}
     \bigl| \Fhat_D(h) - \Fhat_D(0) - \Fhat_D'(0)h \bigr|^2 \ds \\
     &\,\leq\, C_T \|h\|_{C^1(\di D)}^4
     \int_{\sigma+\rmi\R} |s|^{2r+8}
     \|\uihat\|^2_{H^1(\BRtilde\setminus\ol{\BR})} \ds \\
     &\,\leq\, C_T \|h\|_{C^1(\di D)}^4
     \|\ui\|^2_{H^{r+4}_0\bigl(0,T;H^1(\BRtilde\setminus\ol{\BR})\bigr)} 
   \end{split}
 \end{equation*}
 for any $\Rtilde>R$.
 This shows \eqref{eq:EstFrechetDerivativeTDpointwise}, which ends
 the proof.
\end{proof}

The time-domain analog of the operator $\Fhat_{D,R}^\infty$
from~\eqref{eq:OperatorFhatDinfty} is given by
\begin{equation}
  \label{eq:OperatorFDinfty}
  F_D^\infty :\;
  \Dcal(F_D^\infty) \subset C^1(\di D,\Rd) \to H^{r+\frac32}_0(0,T;\R) \,,
  \qquad F_D^\infty(h) \,:=\, \uinfty_{R,h}(\xi,\ph) \,,
\end{equation}
where $\uinfty_{R,h}$ denotes the shifted far field pattern 
of the solution of the time-dependent transmission 
problem~\eqref{eq:TransmissionProblemTD} with~$D$ replaced by $D_h$ 
and with~$\eta = - \gamma\ui|_{\di D_h}$ 
and~$\zeta = -\di_\nu\ui|_{\di D_h}$. 
As before, $\xi\in\Sd$ denotes an arbitrary observation direction and 
$\Dcal(F_D^\infty)$ 
is the same as $\Dcal(\Fhat_D)$ in \eqref{eq:DefDomainFhatD}.

\begin{theorem}
  \label{thm:TemporalDomainDerivativeFarfield}
  The operator $F_D^\infty$ from \eqref{eq:OperatorFDinfty} is Fr\'echet
  differentiable at zero.
  Its Fr\'echet derivative is given by
  $(F_D^\infty)'(0)h = (u_h')_R^\infty(\xi,\ph)$ for all
  $h\in C^1(\di D,\Rd)$, where
  $(u_h')_R^\infty \in H^{r-3}_0(0,T;L^\infty(\Sd))$
  is the far field pattern of the unique solution of the transmission
  problem~\eqref{eq:TransmissionProblemFDTD} with $\eta'$ and $\zeta'$
  from~\eqref{eq:DefEtaZetaTD}. 
  The function $(u'_h)_R^\infty$ is called the temporal domain derivative 
  of the shifted far field pattern. 
  We have
  \begin{equation*}
    \bigl\| F_D^\infty(h) - F_D^\infty(0) - (F_D^\infty)'(0)h \bigr\|_{H^{r}_0(0,T;\R)}
    \,\leq\, C_T \|\ui\|_{H^{r+4}_0\bigl(0,T;H^1(\BRtilde\setminus\ol{\BR})\bigr)}
    \|h\|_{C^1(\di D)}^2 \, .
  \end{equation*}
\end{theorem}

\begin{proof}
 After replacing \eqref{eq:EstFrechetDerivativeLDpointwise} by
 \eqref{eq:EstFrechetDerivativeFarfieldLDpointwise}, the proof is
 analogous to the  proof of
 Theorem~\ref{thm:TemporalDomainDerivativeNearfield}.  
\end{proof}

\subsection{Semi-discretization in time using convolution quadrature}
For the numerical reconstruction algorithms for the inverse backscattering 
problem developed in Section~\ref{sec:NumericalExamples}, we use 
Runge-Kutta convolution quadrature (RKCQ) based on Radau IIA methods for 
time discretization and a boundary element method for spatial discretization.
In this subsection, we apply the Laplace-domain bounds established in 
Sections~\ref{sec:ScatteringProblemLD} and \ref{sec:DomainDerivativeLD} to 
derive convergence rates for the resulting semi-discretization with 
respect to time. 
We refer the reader to \cite[Sec.~5]{BanSay22} for an introduction to 
RKCQ and its numerical implementation.

The Laplace-domain bounds in \eqref{eq:vhatboundpwLD}, 
\eqref{eq:EstFarfieldPatternPointwiseLD}, \eqref{eq:BoundUhat'PointwiseLD}, 
and \eqref{eq:BoundUinftyhat'PointwiseLD} imply exponential decay of 
the pointwise Laplace-domain near and far field evaluations as $\real(s)$ 
increases. 
Such bounds have favorable implications for the convergence of RKCQ time 
discretizations of the associated time-domain convolution-type operators. 
In particular, \cite[Thm.~3]{BanLubMel11} shows that away from the scatterer
the Radau IIA-based convolution quadrature considered in this work attains 
the full classical order of the underlying Runge-Kutta scheme, provided 
that the incoming wave is sufficiently regular in time.

For a given final time $T>0$ and fixed $z \in \Rd\setminus\ol{\BR}$ and $\xi\in S^2$, we aim 
to approximate $\us(z,t)$ and~$u_h'(z,t)$ from~\eqref{eq:usinDp} 
and~\eqref{eq:uprimeinR3}, respectively, as well as $\uinfty_R(\xi,t)$ and 
$(u_h')_R^\infty(\xi)$ from \eqref{eq:uRinfty} and~\eqref{eq:uRprimeinfty}, 
at the discrete time points $t_n=n\tau$, $n=0,\dots,M$, where $\tau>0$ denotes 
the time-step size and, without loss of generality, $M:=T/\tau\in\mathbb{N}$. 
We assume that the temporal approximations are computed using an $m$-stage 
Radau IIA RKCQ method. 
Accordingly, we denote the approximations to $\us_h(z,t_n)$ and $u_h'(z,t_n)$ 
by $(\us_{h,\tau}(z))_n$ and $(u'_{h,\tau}(z))_n$, respectively, and the 
approximations to~$\uinfty_{R,h}(\xi,t_n)$ and $(u_h')_R^\infty(\xi,t_n)$ by 
$(\uinfty_{R,h,\tau}(\xi))_n$ and $((u_h')_{R,\tau}^\infty(\xi))_n$, 
respectively.
For $h=0$, we omit the index $h$ in $u_{h,\tau}^s$ and $u_{R,h}^\infty$.

Combining the Laplace-domain bounds established in Corollary~\ref{cor:LDpwbound} 
and Proposition~\ref{pro:EstFarfieldPatternPointwiseLD} with 
\cite[Thm.~3]{BanLubMel11} yields the following convergence result for the 
solution of the direct scattering problem.

\begin{lemma}
\label{lmm:ErrorTimeDirect}
    Let $z \in \Rd\setminus\ol{\BR}$ and $\xi \in S^2$ be fixed.
    The convolution quadrature semi-discretization in time based on a 
    Radau IIA method with $m$ stages used to approximate both 
    $\us(z,t_n)$ and $\uinfty_R(\xi,t_n)$ provides the following estimates 
    at $t_n = n\tau$:
    \begin{subequations}
    \begin{align}
        \bigl|\us(z,t_n) - (u_{\tau}^s(z))_n\bigr|
        &\,\leq\, C \tau^{2m-1} 
        \bigl\| (\di_\nu \ui, \gamma \ui) 
        \bigr\|_{H_0^r\bigl(0,T;\HminhalfdiD\times\HhalfdiD\bigr)} \,, \label{eq:usconvbound}\\
        \bigl|\uinfty_R(\xi,t_n) - (u_{R,\tau}^\infty(\xi))_n\bigr|
        &\,\leq\, C \tau^{2m-1} 
        \bigl\| (\di_\nu \ui, \gamma \ui) 
        \bigr\|_{H_0^{r-\frac{1}{2}}\bigl(0,T;\HminhalfdiD\times\HhalfdiD\bigr)} \,, \label{eq:uinftyconvbound}
    \end{align}
    \end{subequations}
    for $(\partial_\nu \ui, \gamma \ui) \in H_0^r\bigl(0,T;\HminhalfdiD\times\HhalfdiD\bigl)$ 
    with $r>2m+9/2$.
\end{lemma}

\begin{proof}
    Using Corollary~\ref{cor:LDpwbound} together with the bound 
    from~\cite[Thm.~3]{BanLubMel11}, applied to the causal function~$\ui$ on $\Gamma$, 
    immediately yields the estimate
    \begin{equation*}
        \bigl|\us(z,t_n) - (u_{\tau}^s(z))_n\bigr|
        \,\leq \, C \tau^{2m-1} \int_0^{t_n}
        \| \di_t^{(\widetilde{r}+1)}(\di_\nu \ui, \gamma \ui) 
        \|_{\HminhalfdiD\times\HhalfdiD} \dt'
    \end{equation*}
    for any $\widetilde{r}>2m+3$.  
    We now use the embedding $H_0^r(0,T;X) \subset C_0^{\widetilde{r}+1}([0,T];X)$, 
    valid for any $r>3/2+\widetilde{r}$, to obtain~\eqref{eq:usconvbound}. 
    The error bound~\eqref{eq:uinftyconvbound} follows similarly.
\end{proof}

Using the same approach, we obtain a temporal error bound for the time 
semi-discretization of the temporal domain derivative. 
As before, we derive estimates for both the near and far field, based 
on \cite[Thm.~3]{BanLubMel11} and Proposition~\ref{pro:EstimatesU_h'LD}. 
The proof proceeds analogously to that of Lemma~\ref{lmm:ErrorTimeDirect}.

\begin{theorem}
    Let $z \in \Rd\setminus\ol{\BR}$ and $\xi \in S^2$ be fixed.
    The convolution quadrature semi-discretization in time based on a Radau IIA 
    method with $m$ stages used to approximate both $u_h'(z,t_n)$ and 
    $(u_h')_R^\infty(\xi,t_n)$ provides the following estimates at $t_n = n\tau$:
    \begin{align*}
        \bigl|u_h'(z,t_n) - (u_{h,\tau}'(z))_n\bigr|
        &\,\leq\, C \tau^{2m-1} \bigl\| (\di_\nu \ui, \gamma \ui) 
        \bigr\|_{H_0^r\bigl(0,T;\HhalfdiD\times \HthreehalfdiD\bigr)} 
        \,, \\
        \bigl|(u_h')_R^\infty(\xi,t_n) - ((u_{h}')_{R,\tau}^\infty(\xi))_n\bigr|
        &\,\leq\, C \tau^{2m-1} \bigl\| (\di_\nu \ui, \gamma \ui) 
        \bigr\|_{H_0^{r-\frac{1}{2}}\bigl(0,T;\HhalfdiD\times\HthreehalfdiD\bigr)} \,,
    \end{align*}
    for $(\di_\nu \ui, \gamma \ui) \in H_0^r\bigl(0,T;\HhalfdiD\times \HthreehalfdiD\bigr)$ 
    with $r>2m+8$.
\end{theorem}

A detailed spatial error analysis and, consequently, a comprehensive assessment 
of the full discretization error are beyond the scope of this work. 
To keep the presentation concise, we restrict our discussion to the Galerkin 
implementation and omit further technical details. 
In our numerical experiments, we employ the open-source boundary element method 
library \texttt{bempp}; see \cite{BetScr21} for further details.
Convolution quadrature requires the solution of a large number of scattering 
problems in the Laplace domain. 
This applies both to the computation of the direct problem 
in~\eqref{eq:ExplRepSolnTransmissionProblemLD} and~\eqref{eq:vinftyRrepLD} 
as well as to the evaluation of the domain derivative 
in~\eqref{eq:ExplRepDomainDerivativeLD} and~\eqref{eq:ExplRepDomainDerivativeFarfieldLD}.
We use piecewise linear functions for both the trial and the test function 
space in the Galerkin discretization.
For the Galerkin approximation $\gamma^-\uthat_h$ of $\gamma^-\uthat$ appearing 
on the right-hand side of \eqref{eq:DefEtahatprimeZetahatprime2}, the surface 
gradient $\sgrad \gamma^-\uthat_h$ is piecewise constant. 
Consequently, the surface divergence can be assembled weakly according to
\begin{equation*}
  \int_{\di D} \sdiv \Bigl( h_\nu \Bigl( 1-\frac{1}{\rho_0} \Bigr)
  \sgrad \gamma^-\uthat_h \Bigr) \, \widehat{y}_h \ds
  \,=\, - \int_{\di D} h_\nu \Bigl( 1-\frac{1}{\rho_0} \Bigr)
  \sgrad \gamma^-\uthat_h \cdot \sgrad \widehat{y}_h \ds
\end{equation*}
for all piecewise linear functions $\widehat{y}_h$.

\section{Shape reconstruction for time-dependent backscattering data}
\label{sec:NumericalExamples}
In this section, we return to the inverse scattering problem and
describe an iterative reconstruction algorithm to 
recover the shape of a three-dimensional penetrable scattering object 
from observations time-dependent far field backscattering data at 
finitely many incident/observation directions. 
We suppose that far field patterns $\uinfty(-\theta_k,t)$ as 
in~\eqref{eq:FarfieldTD} corresponding to time-dependent incident 
plane waves $\ui(\ph,\ph;\theta_k)$ as in~\eqref{eq:PlaneWave} with 
direction of propagation~$\theta_k$  are available for finitely many $\theta_1,\ldots,\theta_K\in\Sd$ and all $t\in[0,T]$.
Here the observation time $T>0$ is assumed to be sufficiently large.
In the following we use the notation~$\uinfty(-\theta_k,\ph;\theta_k)$, 
$k=1,\ldots,K$, for such backscattering data obtained with incident 
direction $\theta_k$ and observation direction $-\theta_k$. 

Accordingly, we consider the nonlinear operator $F^\infty$ mapping the shape 
of an obstacle $\di D$ to the associated far field 
data~$(\uinfty(-\theta_k,\ph;\theta_k))_{1\leq k\leq K} \subs L^2(0,T)$.
We restrict ourselves to boundaries~$\di D$ that are star-shaped with
respect to the origin, i.e., $\di D$ can be represented in the parametric 
form
\begin{equation*}
  \di D 
  \,=\, \{ x \in \Rd \;|\; x=r_D(\xi)\xi \,,\; \xi \in \Sd \} 
\end{equation*}
for some positive, twice continuously differentiable 
function~$r_D: \Sd \to \R$ representing the radial distance from the origin. 
For each incident direction $\theta_k$, $k=1,\ldots, K$, the solution 
of the direct scattering problem defines an operator
\begin{equation*}
  F_{\theta_k}^\infty :\; C^2_+(\Sd) \to L^2(0,T) \,, 
\end{equation*}
which maps the parametrization $r_D$ of an admissible obstacle boundary 
$\di D$ to the associated backscattering data~$\uinfty(-\theta_k,\ph;\theta_k)$. 
Here, $C^2_+(\Sd) := \{ r\in C^2(\Sd) \;|\; r>0 \}$.
Therewith, we seek a solution $\di D$ to the nonlinear inverse problem
\begin{equation}
  \label{eq:InverseProblem}
  F_{\theta_k}^\infty(r)
  \,=\, \uinfty(-\theta_k,\ph;\theta_k)
  \qquad \text{for all } k \in \{1, \dots, K\} \,.
\end{equation}
We consider the solution of \eqref{eq:InverseProblem} as the minimum of 
the associated least squares functional.
To handle the ill-conditioning we use a Tikhonov approach, adding 
an~$H^2$-penalty term to obtain 
\begin{equation}
  \label{eq:DefLSQFunctional}
  \Psi_\alpha(r) 
  \,:=\, \sum_{k=1}^{K} \bigl\| F_{\theta_k}^\infty(r) 
  - \uinfty(-\theta_k,\ph;\theta_k) \bigr\|_{L^2(0,T)}^2 
  + \alpha \| r \|_{H^2(\Sd)}^2 \,, \qquad r \in C^2_+(\Sd) \,.
\end{equation} 
As usual, $\alpha>0$ denotes a regularization parameter. 

We apply the Gauß-Newton (GN) method to obtain a sequence of 
approximations $(r_n)_n$ to the parametrization of the obstacle 
boundary $r_D$, typically using a decreasing sequence $(\alpha_n)_n$ 
of regularization parameters depending on the relative residual. 
This is a widely used strategy; see, e.g., \cite{KreRun98}. 
In our numerical implementation, the unknown function~$r$ is approximated 
by a finite linear combination of spherical harmonics of degree at most $L = 5$.
Using the composite trapezoidal rule to discretize the~$L^2(0,T)$-norm 
in \eqref{eq:DefLSQFunctional}, and Parseval's identity to express the 
$H^2(\Sd)$-norm in terms of the $\ell^2$-norm of the spherical harmonics 
coefficients of $r$, the objective functional~$\Psi_\alpha$ can, with a 
slight abuse of notation, be written as
\begin{equation*}
    \Psi_\alpha(p) \,=\, |\Theta(p)|^2 \,.
\end{equation*}
Here, $p\in\R^{(L+1)^2}$ denotes the parameter vector containing 
the spherical harmonics coefficients of~$r$, and $\Theta$ is a 
nonlinear vector-valued function of~$p$.
Denoting the Jacobian of $\Theta$ by $J$, each GN step
\begin{equation*}
    p_{n+1} 
    \,=\, p_n - \bigl( J(p_n)^\trans J(p_n) \bigr)^{-1} 
    J(p_n)^\trans \Theta(p_n) \,, \qquad n\in\N \,,
\end{equation*}
requires the computation of the backscattered waves 
$F_{\theta_k}^\infty(r_n)$ and their temporal domain derivative 
of $(F_{\theta_k}^\infty)'(r_n)$ for each incident 
direction $\theta_k$, $k=1,\ldots,K$.

\begin{figure}[th]
  \centering
  \begin{subfigure}{0.32\textwidth}
    \centering
    \includegraphics[width=1.\linewidth]{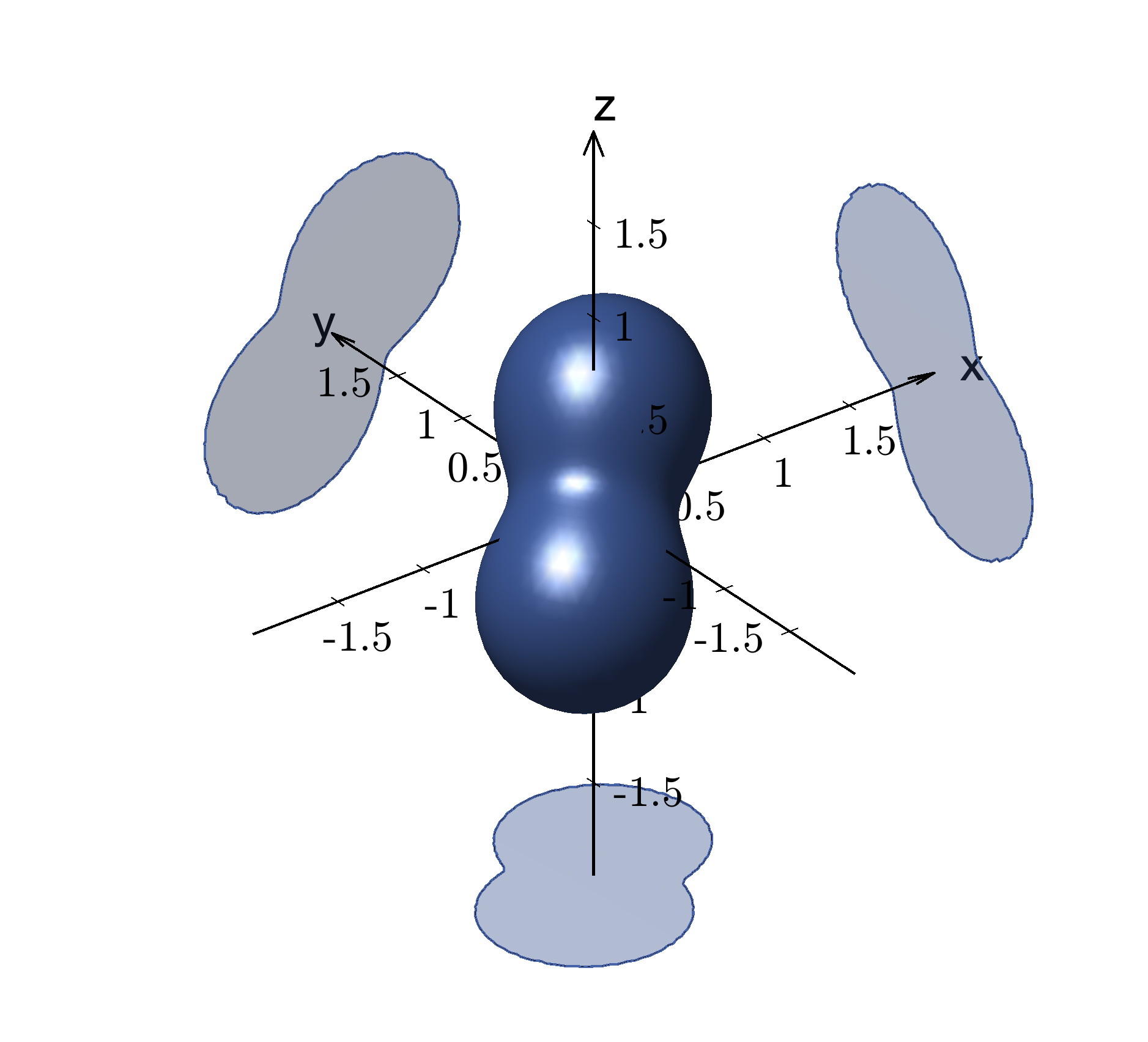}
  \end{subfigure}
  \begin{subfigure}{0.32\textwidth}
    \centering
    \includegraphics[width=1.\linewidth]{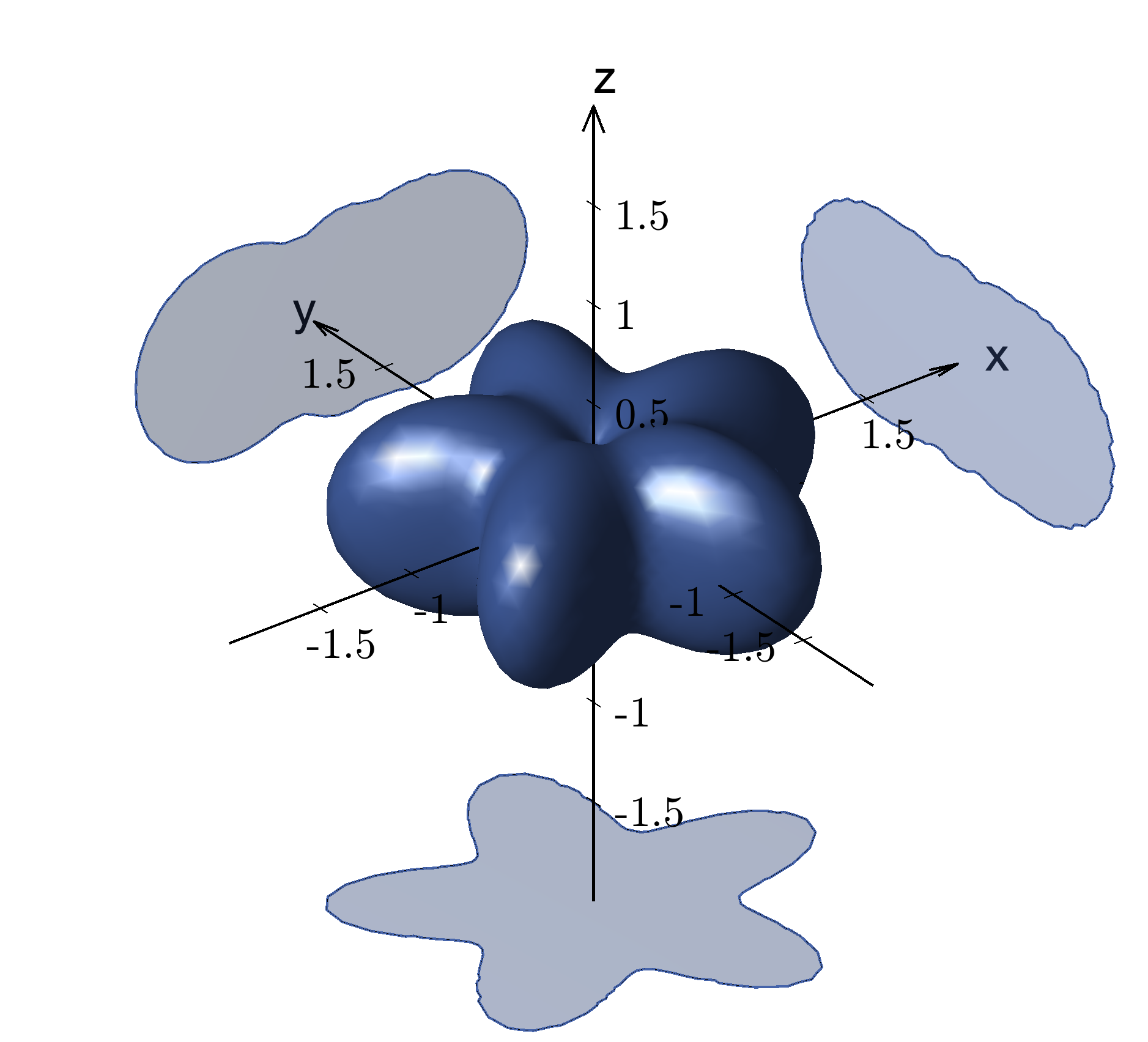}
  \end{subfigure}
  \begin{subfigure}{0.32\textwidth}
    \centering
    \includegraphics[width=1.\linewidth]{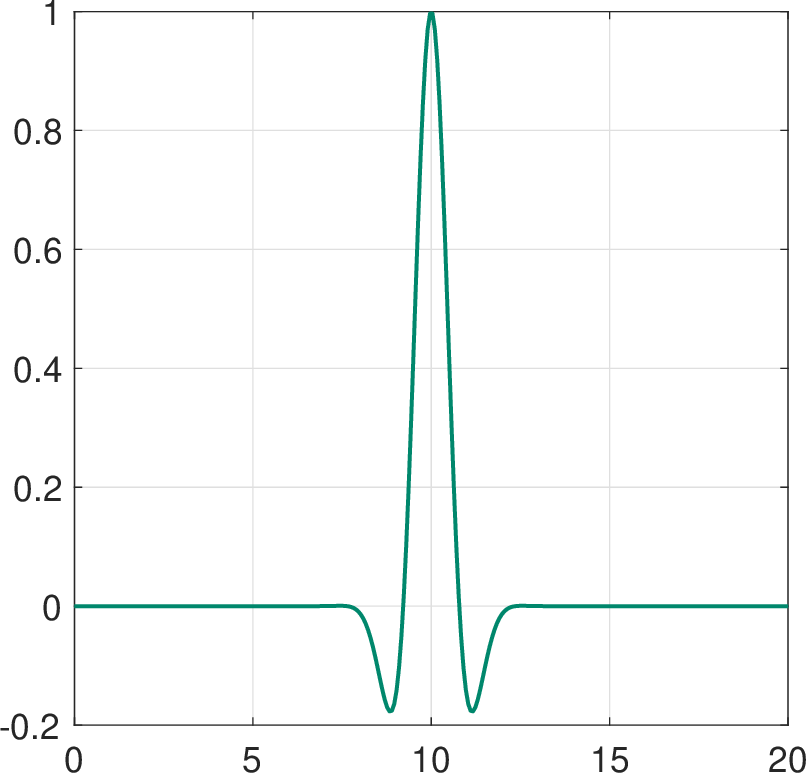}
  \end{subfigure}
  \caption{Exact shape of the scatterers in Examples~\ref{ex:TiltedPeanutInverse} (left) and~\ref{ex:seastarinverse} (middle). Plot of profile $f^{i}$ of incident plane wave as function of time (right).}
  \label{fig:ExactShape}
\end{figure}

In our numerical examples below we consider the two different 
obstacles shown in Figure~\ref{fig:ExactShape}: a tilted 
peanut-shape (left) and a sea star with five lobes (middle).
We note that these exact shapes do not belong to the finite 
dimensional shape space that is used for the reconstruction.
For both obstacles, the density inside the obstacle is chosen 
as~$\rho_0 = 0.1$ and the wave speed inside the obstacle 
as~$c_0 = 0.05$. 
The triangular mesh that we use to discretize the boundaries 
of these obstacles in the numerical solution of the forward 
problem to generate synthetic forward data has $1610$ vertices. 

We consider incident plane waves as in \eqref{eq:PlaneWave} 
given by 
\begin{equation*}
  \ui(x,t;\theta_k)
  \,=\, f^i(t-\theta_k\cdot x) \,, \qquad 
  (x,t)\in \Rd\times[0,T] \,, 
\end{equation*}
for $k=1,\ldots,K$, where
\begin{equation*}
  f^i(t) 
  \,=\, \chi(t-10) \cos(-20+2t) \rme^{-(t-10)^2} \,, 
  \qquad t\in\R \,.
\end{equation*}
Here $\chi$ is a smooth cut-off function satisfying $\chi=1$ 
in $[-4,4]$ and $\chi=0$ in $\R\setminus (-5,5)$, ensuring that 
$f^i$ is compactly supported, in accordance with the assumptions 
stated above.
For the final observation time we choose $T=20$ in both examples,
using $M = 300$ equidistant time steps in~$[0,T]$ in the numerical 
discretizations.
The choice of this incoming wave is similar as in~\cite{Ban10} 
and represents a plane wave modulated by a Gaussian.
A plot of $f^i$ is shown in Figure~\ref{fig:ExactShape}~(right). 
We note that both scattering obstacles are contained 
in a ball~$\BR$ of radius~$R=2$ around the origin. 
Accordingly, the incident waves $\ui(\ph,\ph;\theta_k)$, 
$k=1,\ldots, K$, clearly satisfy the causality 
condition~\eqref{eq:CausalityUi}. 

In our examples we choose three different sets of backscattering 
directions for the reconstruction. 
These are given by
\begin{itemize}
\item[(i)] $\Theta_6 := \{ \pm e_1, \pm e_2, \pm e_3 \}$, i.e., K=6\,, 
\item[(ii)] $\Theta_8 := \bigl\{ a_1 e_1 + a_2 e_2 + a_3 e_3 
  \;\big|\; a_l \in \bigl\{ \pm 1 / \sqrt{3} \bigr\} \,,\; 
  l=1,2,3 \bigr\}$, i.e., K=8\,, 
\item[(iii)] $\Theta_{14} := \Theta_6 \cup \Theta_8$, i.e., K=14\,, 
\end{itemize}
where as usual $e_1$, $e_2$, and $e_3$ denote the standard basis 
vectors in~$\Rd$. 

For the initial guess in the GN iteration we use a ball 
of radius $0.3$ centered at the origin, as shown in 
Figure~\ref{fig:Reconstruction_peanut_shifted}~(top left).
This means $r_0\equiv 0.3$.
The triangular meshes that we use to discretize the boundaries 
of this initial guess $r_0$ and all further iterates $r_n$, 
$n\in\N$, in the GN iteration, when computing the derivative of 
the regularized objective functional~$\Psi_{\alpha_n}$ in each 
iteration step, have~$1026$ vertices. 

The stopping rule for the GN iteration is determined by the 
relative residual
\begin{equation*}
  \mathrm{Res(n)}
  \,:=\, \frac{ \Bigl( \sum_{k=1}^K 
    \| F_{\theta_k}^\infty(r_n) 
    - u^\infty_{\theta_k}(-\theta_k,\ph) \|^2_{L^2(0,T)} 
    \Bigr)^{\frac12} }
  { \Bigl( \sum_{k=1}^K 
    \| F_{\theta_k}^\infty(r_n) \|^2_{L^2(0,T)} \Bigr)^{\frac12} } \,, 
  \qquad n\in\N \,,
\end{equation*}
and by the relative update
\begin{equation*}
  \mathrm{Up(n)}
  \,:=\, \frac{\|r_{n+1}-r_n\|_{L^2(\Sd)}}{\|r_n\|_{L^2(\Sd)}}  \,, 
  \qquad n\in\N \,,
\end{equation*}
in the $n$th step of the iteration. 
We stop the GN iteration, when either the relative 
residual~$\mathrm{Res(n)}$ falls below~$0.01$ or if the 
relative update $\mathrm{Up(n)}$ falls below~$9\cdot 10^{-4}$. 

\begin{example}
  \label{ex:TiltedPeanutInverse}

  \begin{figure}[th]
    \centering

    \begin{subfigure}{0.45\textwidth}
      \centering
      \includegraphics[width=1.\linewidth]{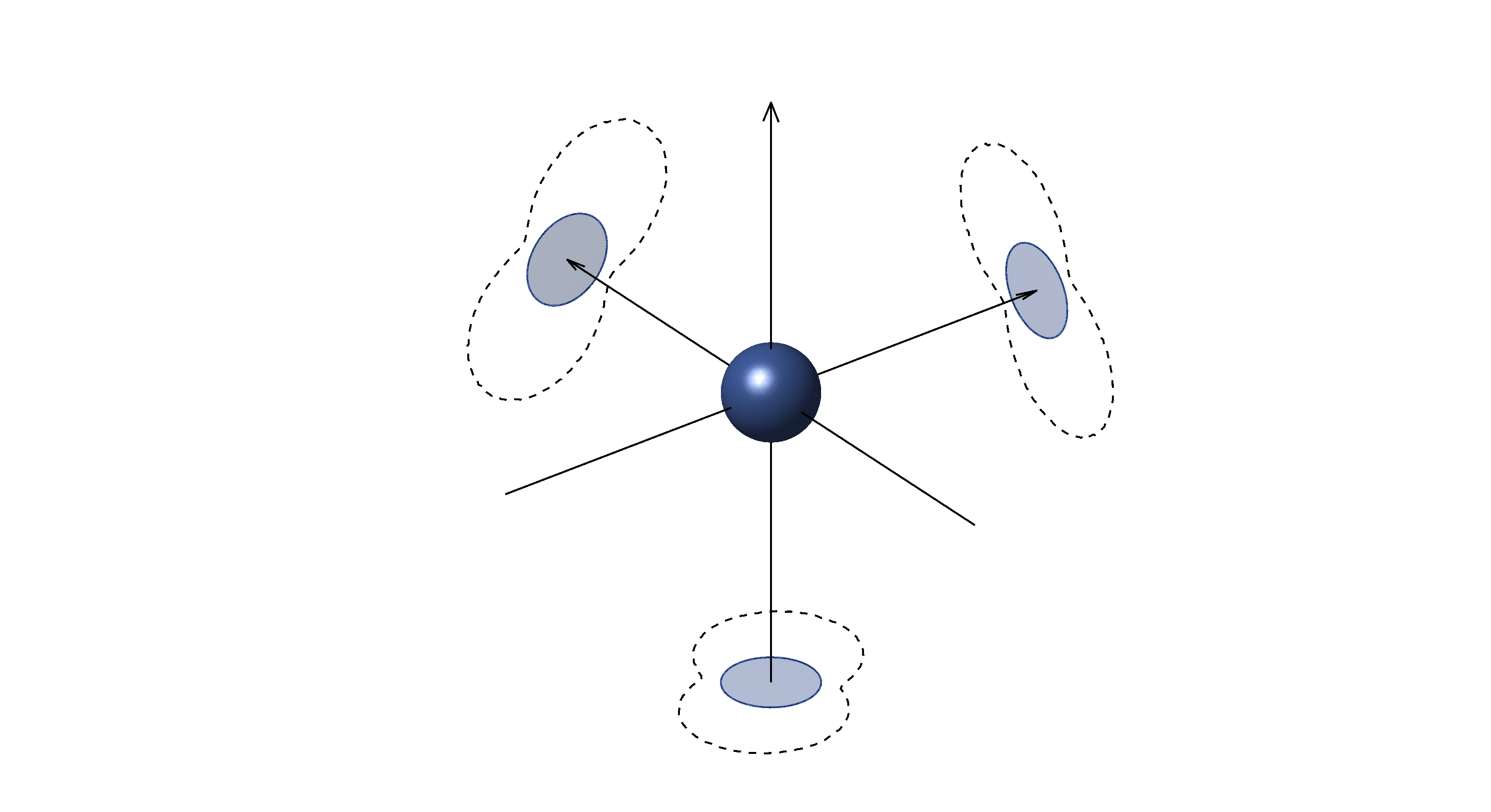}
      \caption{Initial guess}
    \end{subfigure}
    \hfill
    \begin{subfigure}{0.45\textwidth}
      \centering
      \includegraphics[width=1.\linewidth]{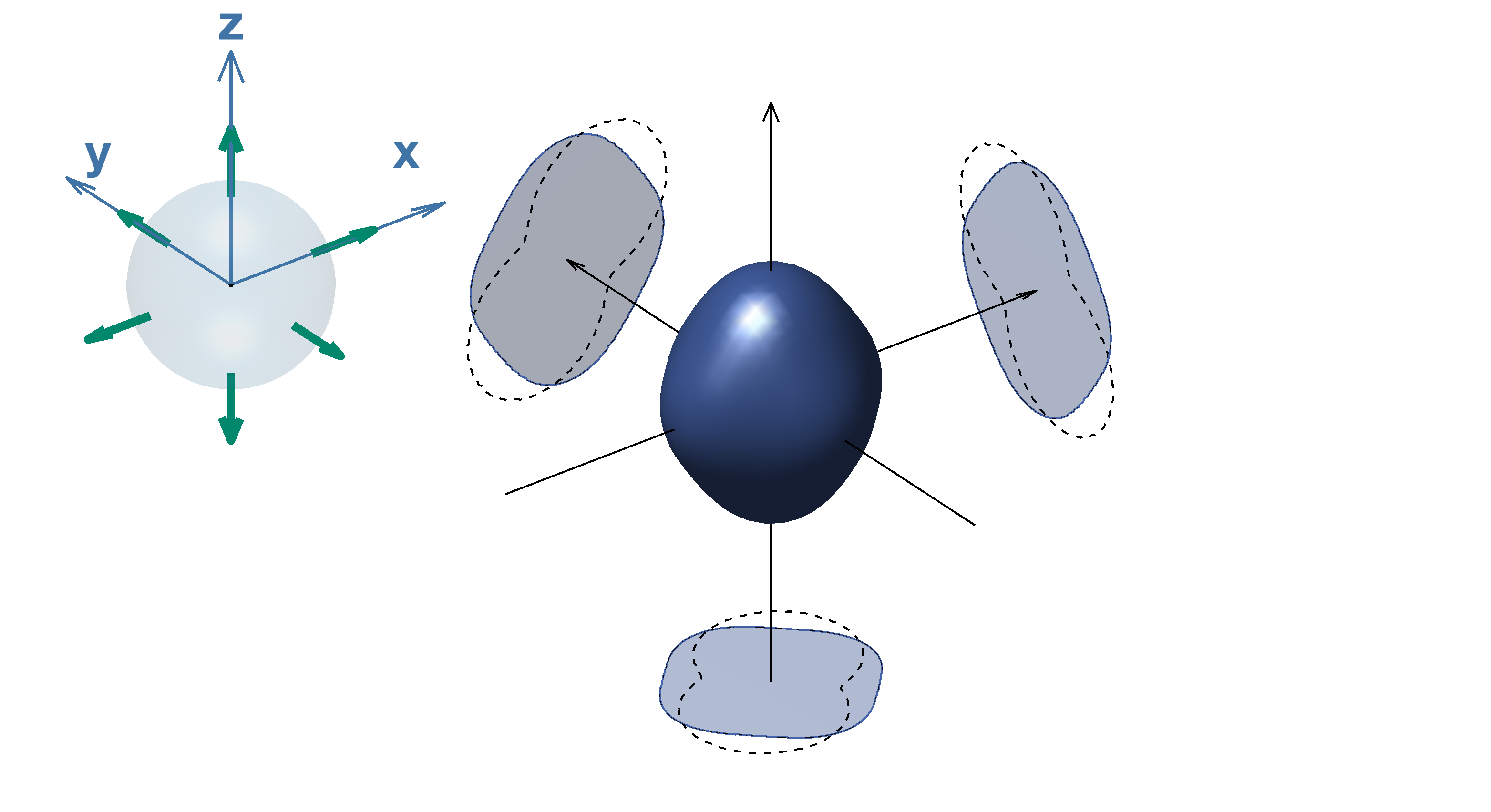}
      \caption{Reconstruction ($6$ backscattering directions)}
    \end{subfigure}

    \vspace{0.5cm}

    \begin{subfigure}{0.45\textwidth}
      \centering
      \includegraphics[width=1.\linewidth]{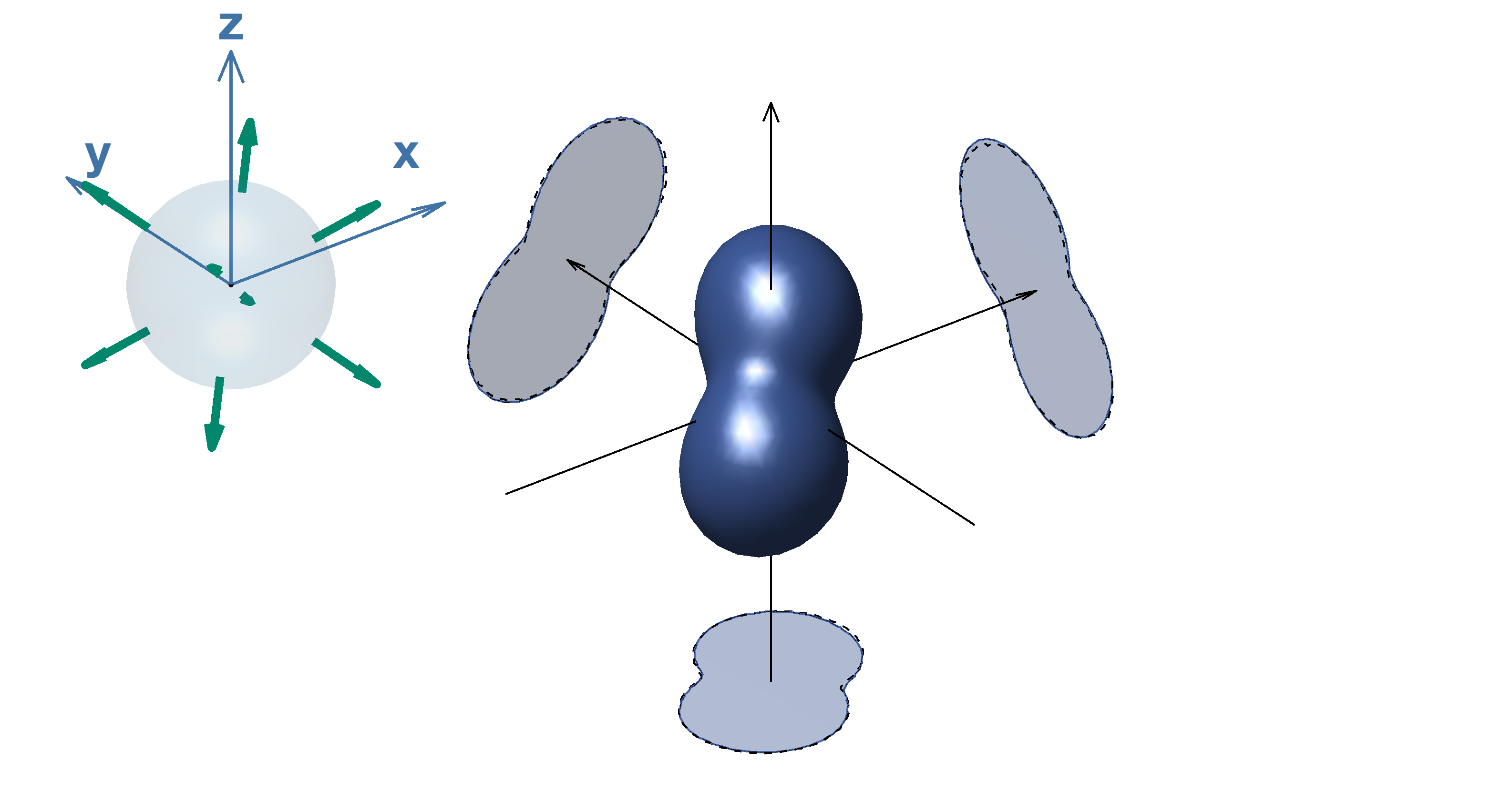}
      \caption{Reconstruction ($8$ backscattering directions)}
    \end{subfigure}
    \hfill
    \begin{subfigure}{0.45\textwidth}
      \centering
      \includegraphics[width=1.\linewidth]{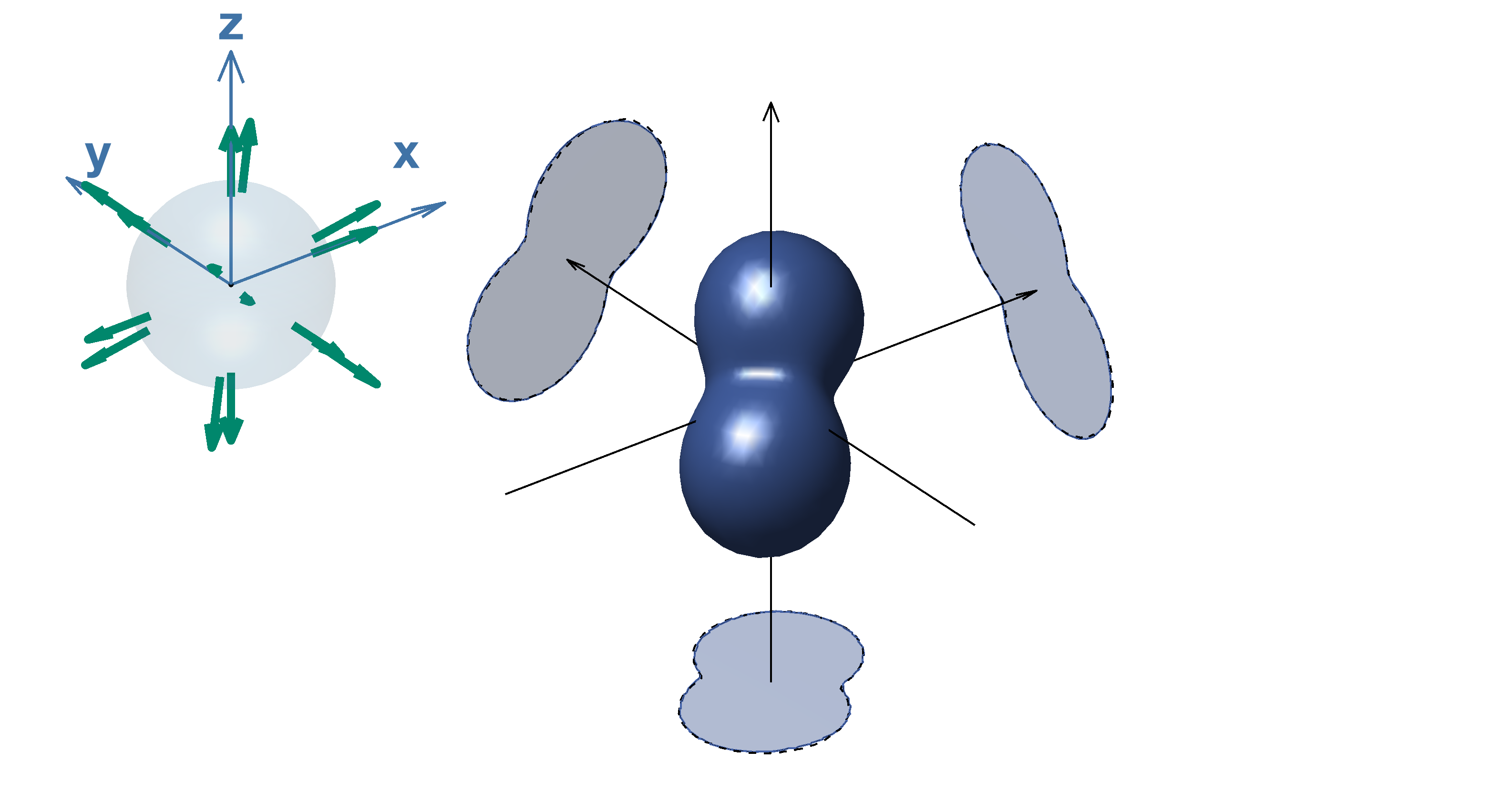}
      \caption{Reconstruction ($14$ backscattering directions)}
    \end{subfigure}  
    \caption{Reconstruction of tilted peanut-shape from backscattering data 
    corresponding to $6$, $8$ and $14$ backscattering directions.
    Spheres and arrows to the left of the plots (b)--(d) indicate 
    incident/backscattering directions.}
    \label{fig:Reconstruction_peanut_shifted}
  \end{figure}

  In our first example we consider the peanut-shaped scattering 
  object as shown in Figure~\ref{fig:ExactShape} (left).
  The initial guess used in the GN iteration is shown in 
  Figure~\ref{fig:Reconstruction_peanut_shifted} (top left). 
  During the GN iteration the regularization parameter is set 
  to $\alpha_n = 5\cdot 10^{-4}$ as long 
  as $\mathrm{Res}(n) \geq 0.1$. 
  As soon as~$\mathrm{Res}(n) < 0.1$,  we choose to reduce the 
  regularization by setting the regularization parameter to 
  $\alpha_n = 10^{-5}$.
  These values were established through trial and error.

  \begin{figure}[th]
    \centering
    \includegraphics[width=1.\linewidth]{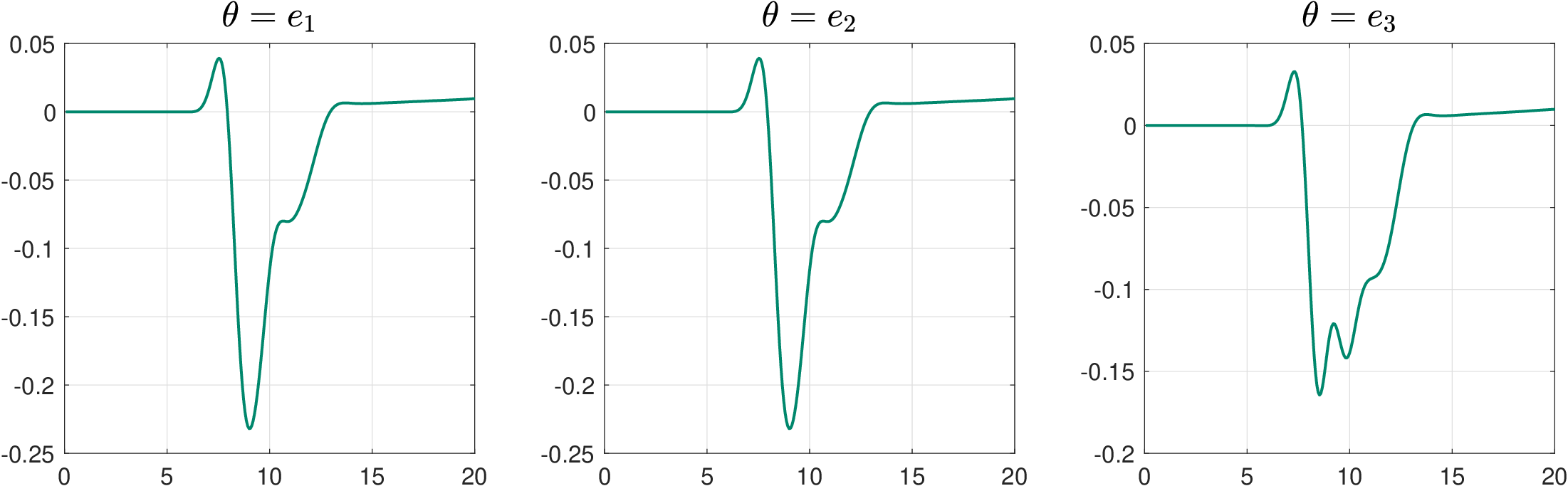}
    \caption{Backscattering data as function of time for $3$ different backscattering 
      directions in Example~\ref{ex:TiltedPeanutInverse}.}
    \label{fig:farField_3dir_peanut}
  \end{figure}
  We start with the $K=6$ backscattering directions in $\Theta_6$ from case~(i).
  Plots of the corresponding backscattering data 
  $\uinfty(-e_1,t;e_1)$, $\uinfty(-e_2,t;e_2)$, and 
  $\uinfty(-e_3,t;e_3)$ for $t\in[0,20]$ are shown 
  in~Figure~\ref{fig:farField_3dir_peanut}. 
  Here the left plot and the plot in the middle are the same due 
  to symmetry properties of the scattering obstacle. 
  The initial relative residual satisfies $\mathrm{Res}(0)=0.78$.
  The GN iteration terminates after $6$ iterations and the corresponding
  relative residual is at $\mathrm{Res}(6)=0.17$. 
  The reconstruction is shown in 
  Figure~\ref{fig:Reconstruction_peanut_shifted}~(top right). 
  In our second test we use the $K=8$ backscattering directions 
  in $\Theta_8$ from case~(ii).
  Here, the initial relative residual lies at $\mathrm{Res}(0)=0.85$.
  The GN iteration stops after $5$ steps, and the corresponding
  relative residual is at $\mathrm{Res}(5)=0.01$. 
  The reconstruction is shown in 
  Figure~\ref{fig:Reconstruction_peanut_shifted}~(bottom left). 
  When using all $K=14$ backscattering directions in $\Theta_{14}$
  from case~(iii), we obtain an initial relative residual 
  of~$\mathrm{Res}(0)=0.82$.
  The GN iteration terminates after $6$ steps, and the corresponding
  relative residual is at $\mathrm{Res}(6)=0.01$. 
  The reconstruction is shown in 
  Figure~\ref{fig:Reconstruction_peanut_shifted}~(bottom right). 

  Each final reconstruction in 
  Figure~\ref{fig:Reconstruction_peanut_shifted} is shown alongside a 
  sphere and direction vectors (top left of each reconstruction) which 
  visualize the backscattering directions $\Theta_6$, $\Theta_8$, and 
  $\Theta_{14}$, respectively, that were used for the reconstruction. 
  Furthermore, the plots include the projections of the reconstructions 
  onto the three coordinate planes. 
  For comparison, the corresponding projections of the exact scatterer 
  are shown as dotted lines.

  For this example, we observe that reconstructions using $8$ and $14$
  backscattering directions yield good results. 
  In contrast, $6$ backscattering directions appear to be insufficient 
  for achieving a satisfactory reconstruction.
  \hfill$\lozenge$
\end{example}

\begin{example}
  \label{ex:seastarinverse}

  \begin{figure}[th]
    \centering
    
    \begin{subfigure}{0.45\textwidth}
      \centering
      \includegraphics[width=1.\linewidth]{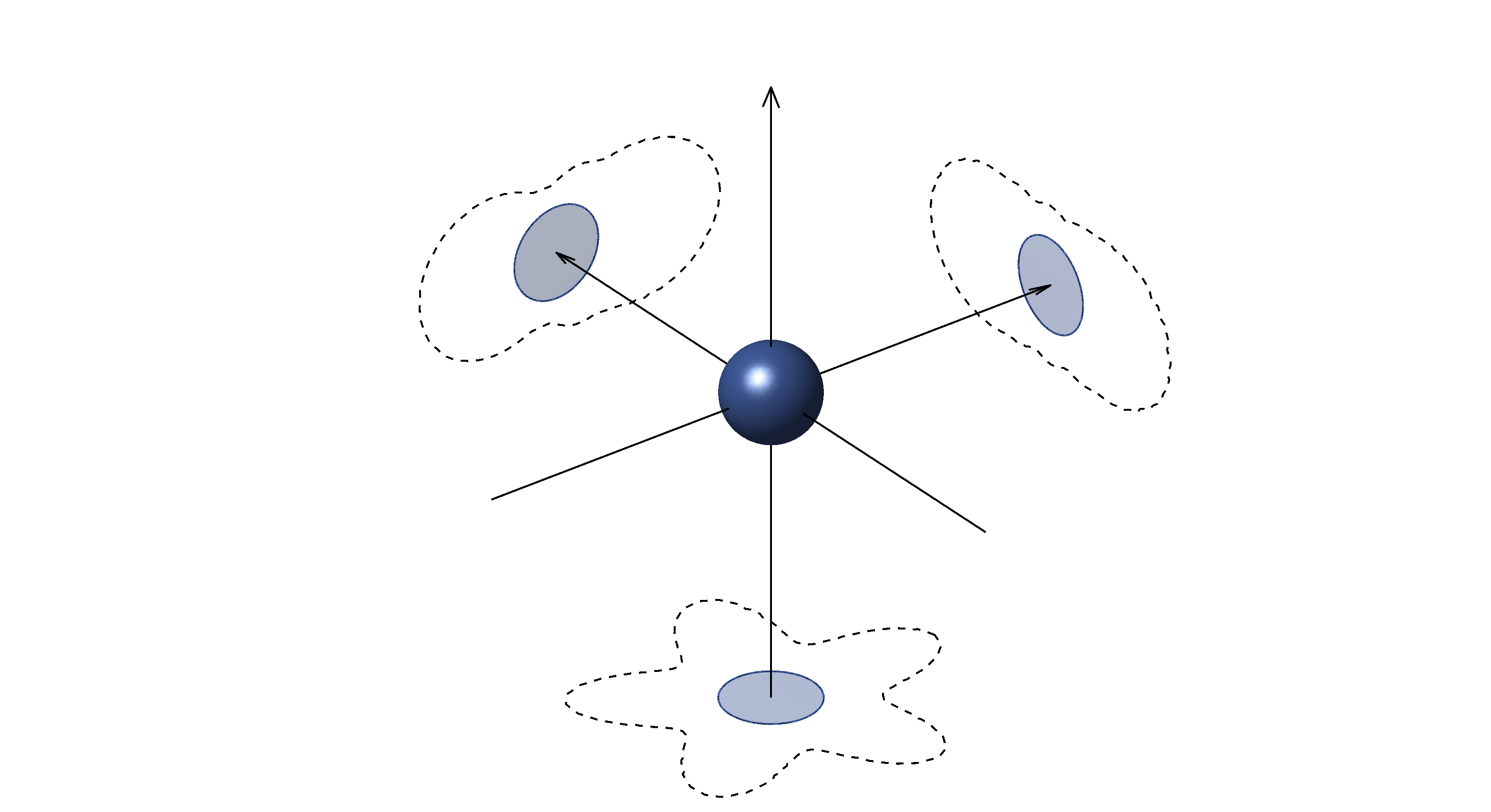}
      \caption{Initial guess}
    \end{subfigure}
    \hfill
    \begin{subfigure}{0.45\textwidth}
      \centering
      \includegraphics[width=1.\linewidth]{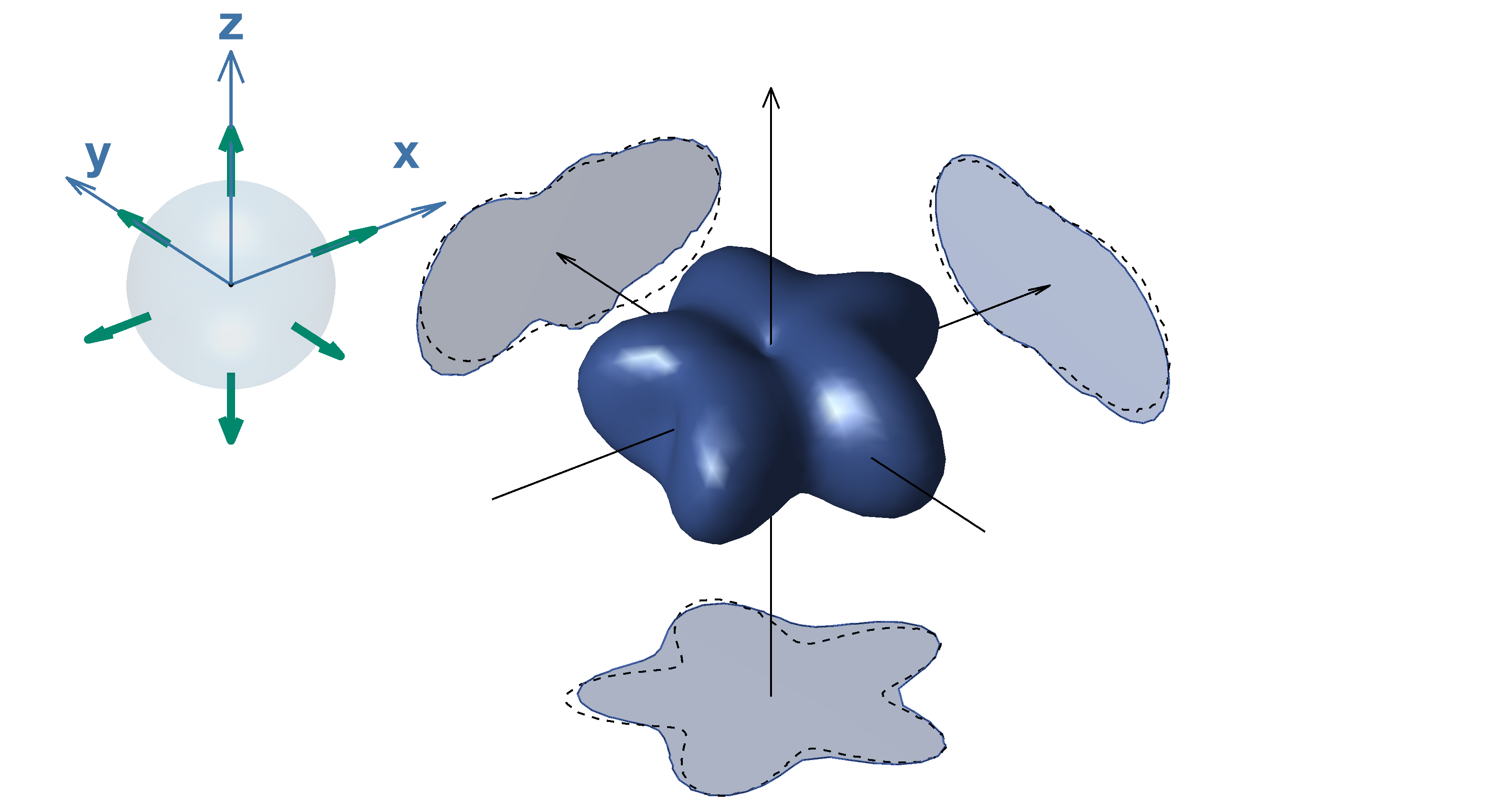}
      \caption{Reconstruction (6 backscattering directions)}
    \end{subfigure}
    
    \vspace{0.5cm}
    
    \begin{subfigure}{0.45\textwidth}
      \centering
      \includegraphics[width=1.\linewidth]{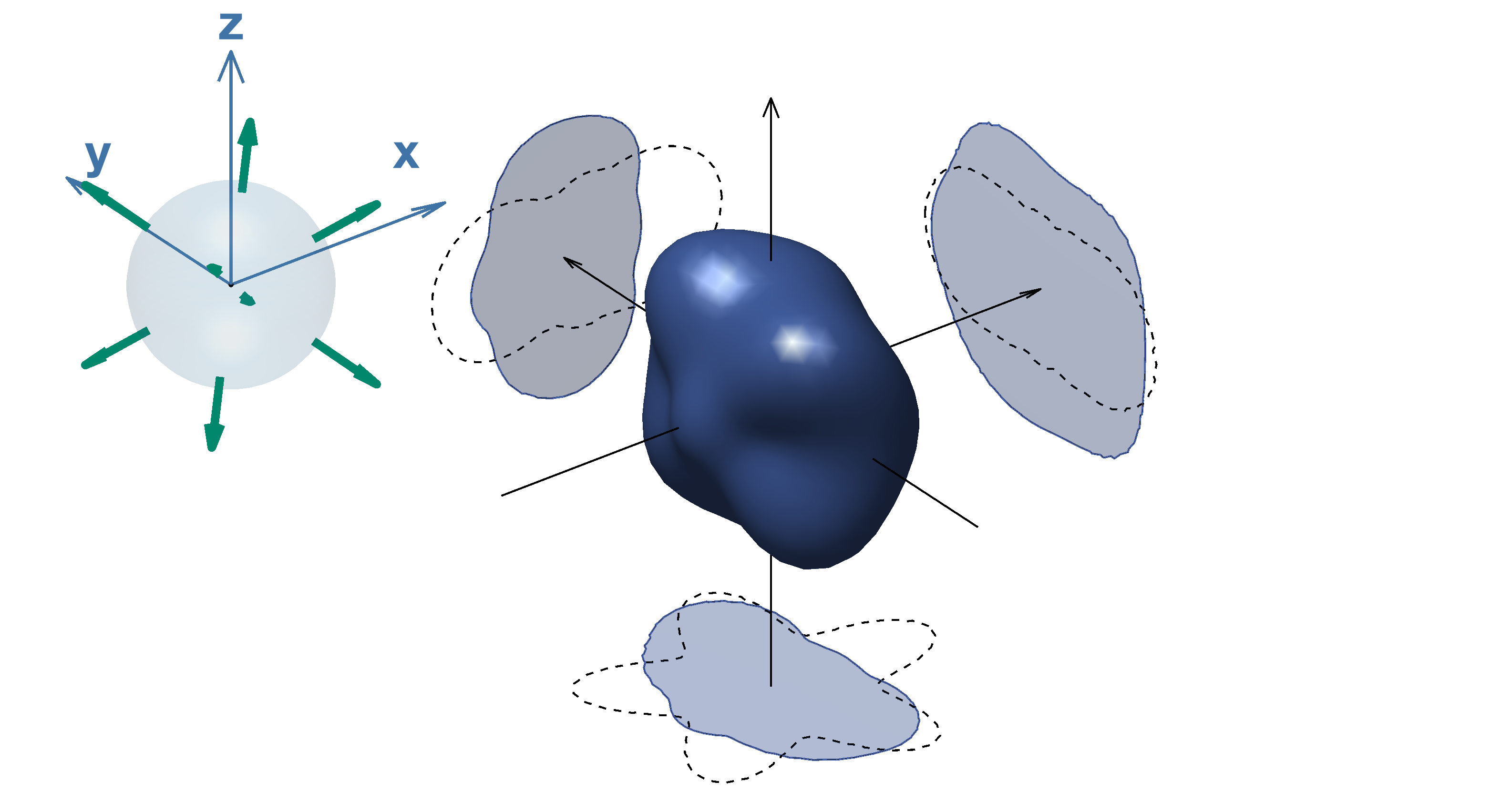}
      \caption{Reconstruction (8 backscattering directions)}
    \end{subfigure}
    \hfill
    \begin{subfigure}{0.45\textwidth}
      \centering
      \includegraphics[width=1.\linewidth]{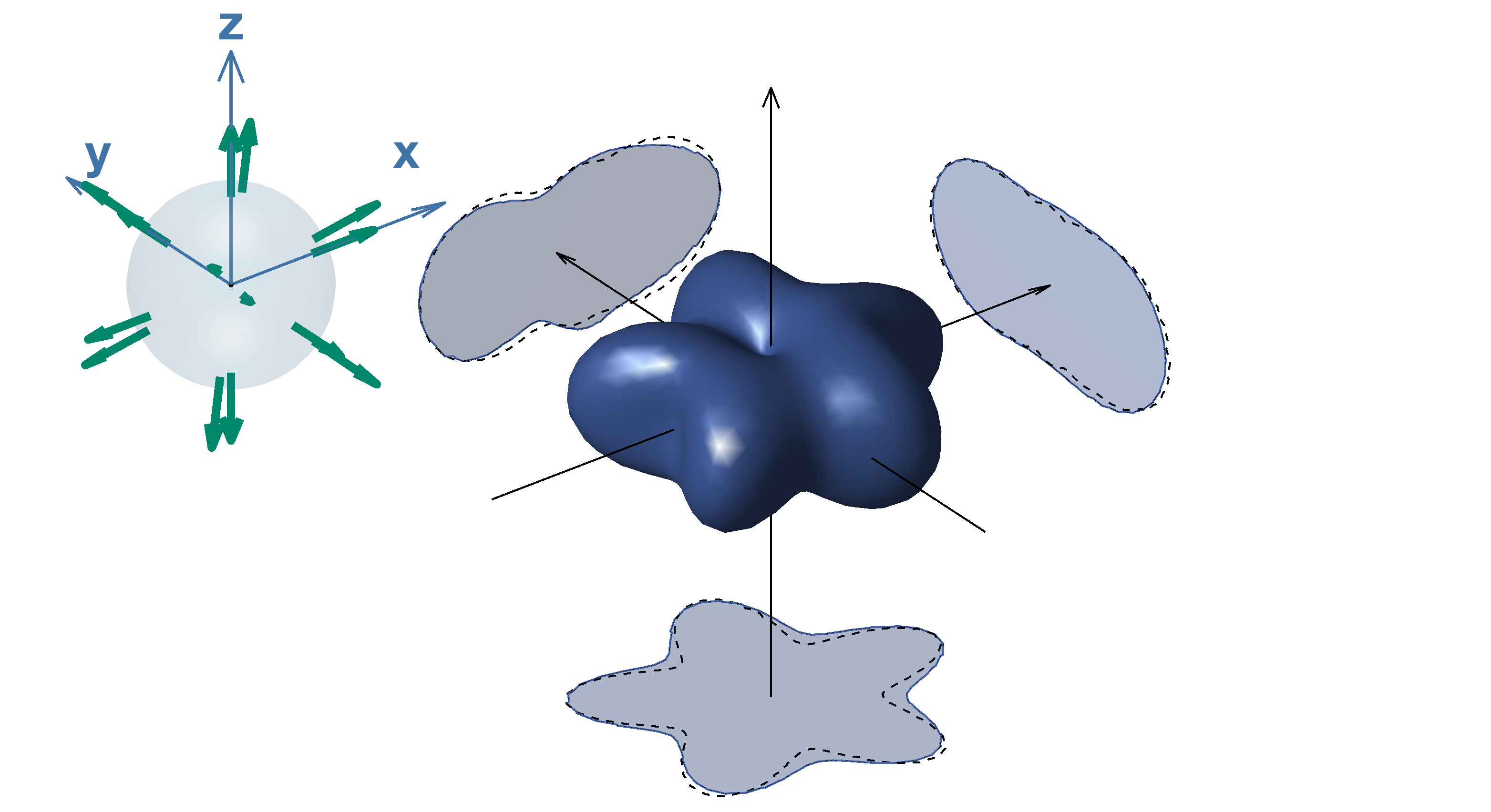}
      \caption{Reconstruction (14 backscattering directions)}
    \end{subfigure}
    
    \caption{Reconstruction of sea star-shape from backscattering data 
      corresponding to $6$, $8$ and $14$ backscattering directions.
      Spheres and arrows to the left of the plots (b)--(d) indicate 
      incident/backscattering directions.}
    \label{fig:Reconstruction_star}
  \end{figure}
  
  In our second example we consider the sea star-shaped scattering 
  object, as shown in Figure~\ref{fig:ExactShape} (right).
  As in the previous Example~\ref{ex:TiltedPeanutInverse}, we use 
  a ball of radius $0.3$, centered at the origin, shown in 
  Figure~\ref{fig:Reconstruction_star} (top left) for the initial 
  guess.
  For relative residuals $\mathrm{Res}(n) \geq 0.1$, the regularization 
  parameter is set to $\alpha_n = 9\cdot 10^{-5}$ in the GN iteration. 
  As soon as~$\mathrm{Res}(n) < 0.1$,  we switch to $\alpha_n = 10^{-5}$.
  Again, these values were established through trial and error.

  For the $K=6$ backscattering directions in $\Theta_6$, the initial relative 
  residual satisfies $\mathrm{Res(0)} = 0.97$.
  The GN iteration terminates after~$15$ steps and the corresponding
  relative residual is at $\mathrm{Res}(15)=0.02$. 
  The reconstruction is shown in 
  Figure~\ref{fig:Reconstruction_star}~(top right). 
  When using the $K=8$ backscattering directions in~$\Theta_8$, the initial 
  relative residual lies at $\mathrm{Res(0)}=0.93$.
  The GN iterations stops after~$15$ steps, and the corresponding relative 
  residual is at~$\mathrm{Res}(15)=0.05$. 
  The reconstruction is shown in 
  Figure~\ref{fig:Reconstruction_star}~(bottom left). 
  When using all $14$ directions in $\Theta_{14}$, we obtain an initial 
  relative residual of $\mathrm{Res}(0)=0.95$. 
  The GN iteration terminates after after $11$ steps, and the corresponding
  relative residual is at $\mathrm{Res}(11)=0.02$. 
  The reconstruction is shown in 
  Figure~\ref{fig:Reconstruction_star}~(bottom right). 

  For this example, only the reconstruction with $14$ backscattering directions 
  produces a satisfactory result. 
  In contrast, using $6$ or $8$ backscattering directions does not provide 
  sufficient information for an accurate reconstruction.
  \hfill$\lozenge$
\end{example}

\begin{example}

  \begin{figure}[th]
    \centering

    \begin{subfigure}{0.32\textwidth}
      \centering
      \includegraphics[width=1.\linewidth]{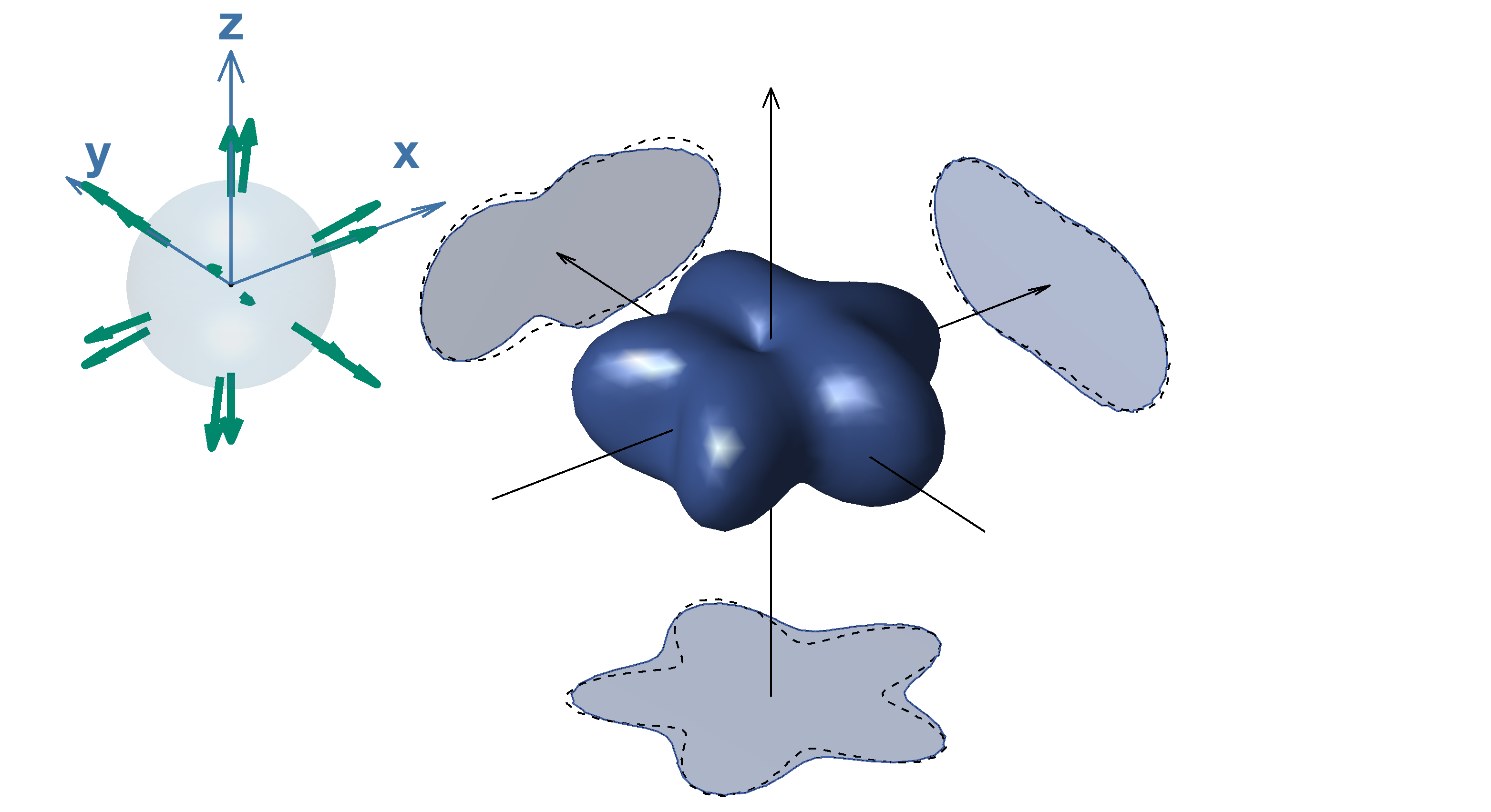}
      \caption{Reconstruction (14 backscattering directions and $5 \%$ noise)}
    \end{subfigure}
    \hfill
    \begin{subfigure}{0.32\textwidth}
      \centering
      \includegraphics[width=1.\linewidth]{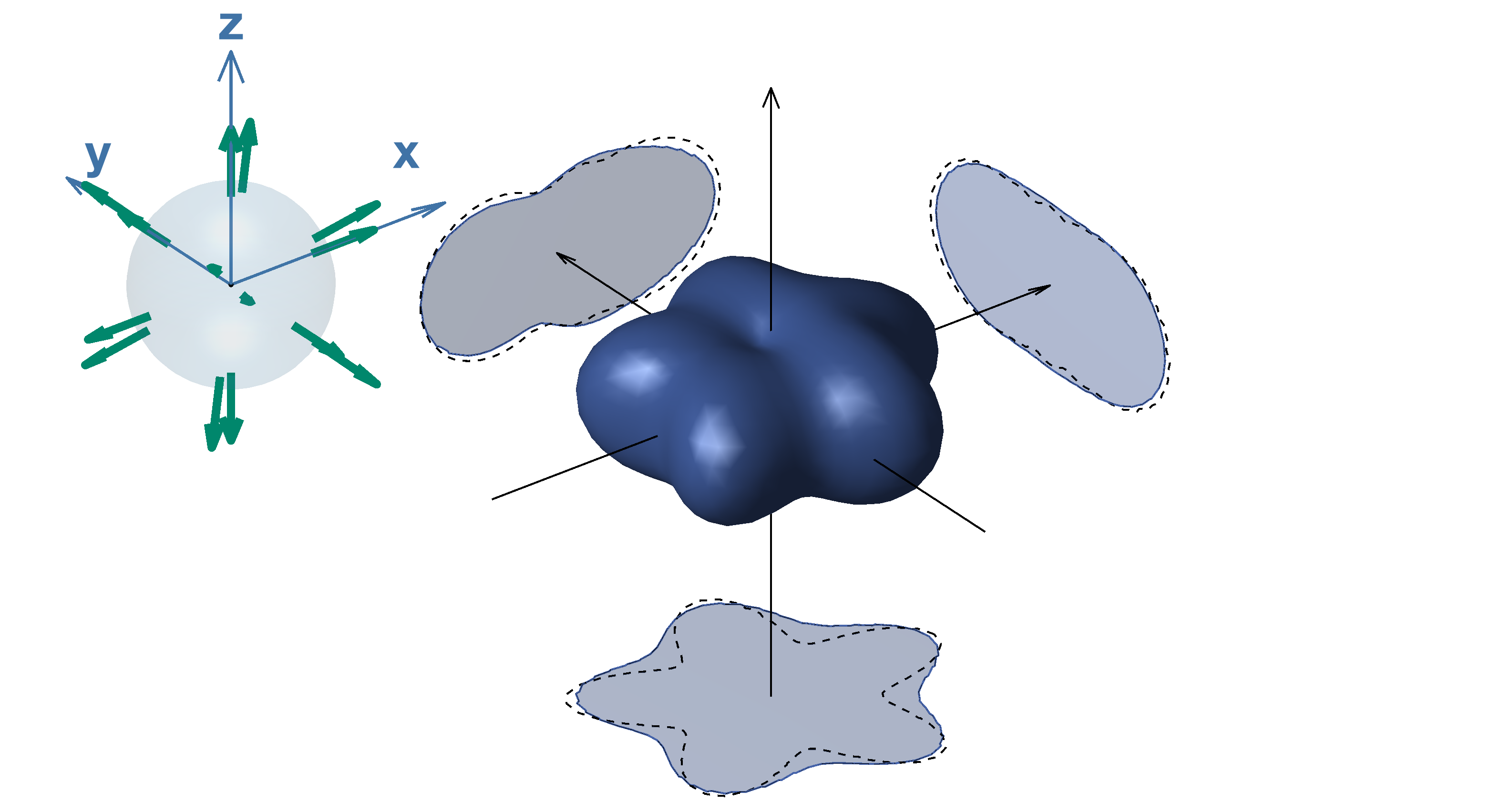}
      \caption{Reconstruction (14 backscattering directions and $15 \%$ noise)}
    \end{subfigure}
    \hfill
    \begin{subfigure}{0.32\textwidth}
      \centering
      \includegraphics[width=1.\linewidth]{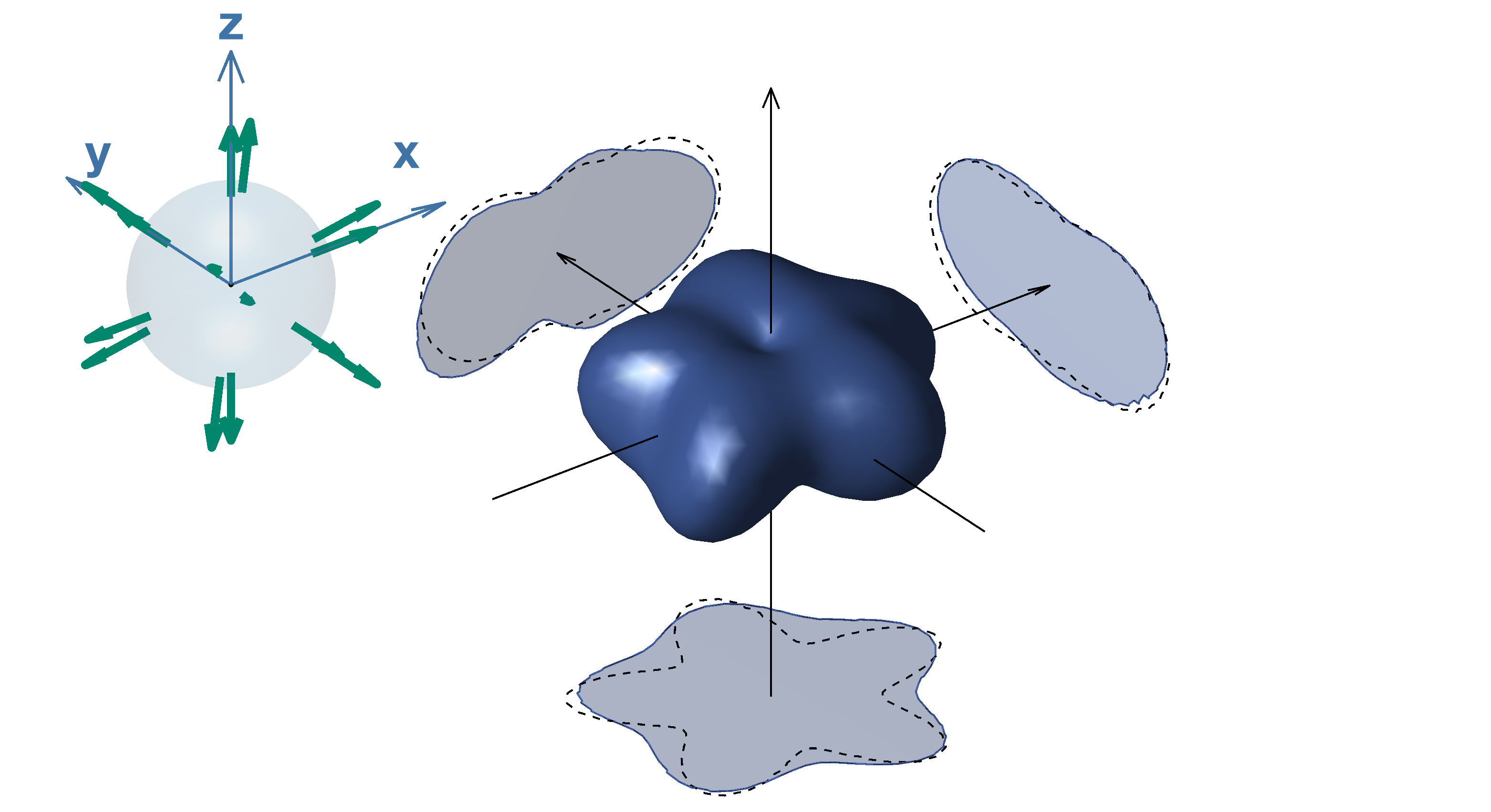}
      \caption{Reconstruction (14 backscattering directions and $30 \%$ noise)}
    \end{subfigure}

    \caption{Reconstruction of sea star-shape from noisy backscattering data 
    with 14 backscattering directions.
    Spheres and arrows to the left of the plots indicate incident/backscattering 
    directions.}
    \label{fig:Reconstruction_star_noisy}
  \end{figure}

  To assess the sensitivity of the reconstruction method to noise in the data, we repeat 
  the previous examples using the $14$ backscattering directions in~$\Theta_{14}$. 
  The backscattering data are corrupted with additive uniformly distributed random noise 
  at levels of $5\%$, $15\%$, and $30\%$, respectively.
  We use the same initial guess and the same regularization parameters as in Examples~\ref{ex:seastarinverse}. 

  For a noise-level of $5\%$, we obtain an initial relative residual of 
  $\mathrm{Res(0)}=0.95\%$.
  The GN iteration terminates after $11$ steps and the corresponding relative 
  residual is at $\mathrm{Res}(11)=0.06$.
  The reconstruction is shown in Figures~\ref{fig:Reconstruction_star_noisy} (left). 
  When setting the noise-level to $15\%$, the initial relative residual lies at 
  $\mathrm{Res}(0)=0.95$.
  The GN iteration terminates after $9$ steps and the corresponding relative 
  residual is $\mathrm{Res}(9)=0.16$.
  The reconstruction is shown in Figures~\ref{fig:Reconstruction_star_noisy} (middle). 
  After increasing the noise-level to $30\%$, we obtain an initial relative 
  residual of $\mathrm{Res}(0)=0.96$.
  The GN iteration stops after $9$ steps and the corresponding relative residual is 
  at $\mathrm{Res}(9)=0.29$.
  The reconstruction is shown in Figures~\ref{fig:Reconstruction_star_noisy} (right). 

  For data contaminated with $5\%$ noise, the reconstruction is almost as accurate 
  as in the case of exact data. 
  Even for noise levels of $15\%$ and $30\%$, the reconstructions remain satisfactory. 
  These results indicate that the proposed reconstruction method performs reasonably 
  well in the presence of noise.~\hfill$\lozenge$
\end{example}

\begin{appendix}
  \section*{Appendix}
  %%%%%%%%%%%%%%%%%%%%%%%%%%%%%%%%%%%%%%%%%%%%%%%%%%%%%%%%%%%%%%%%%%%%%%%% 
  \section{Estimates for the modified Calder\'on operator $\Ccal(s)$}
  \label{app:A}
  %%%%%%%%%%%%%%%%%%%%%%%%%%%%%%%%%%%%%%%%%%%%%%%%%%%%%%%%%%%%%%%%%%%%%%%% 
  In this section we show that the operator $\Ccal(s))$ from
  \eqref{eq:DefOpCcal} is boundedly invertible, and we
  establish~\eqref{eq:EstInverseOpCcal}.  
  Since 
  \begin{equation*}
    \Ccal(s)
    \,=\, \frac{1}{c_0\rho_0} \begin{bmatrix}
      c_0\rho_0 & 0 \\ 0 & 1
    \end{bmatrix}
    \Bcal(s_0)
    \begin{bmatrix}
      c_0\rho_0 & 0 \\ 0 & 1
    \end{bmatrix}
    + \Bcal(s) \,,
  \end{equation*}
  where
  \begin{equation*}
    \Bcal(s) 
    \,:=\, \begin{bmatrix}
        s\SLop(s) 
        & \DLop(s) \\
        - \adjDLop(s) 
        & \frac{1}{s}\hypSingop(s) 
      \end{bmatrix}
  \end{equation*}
  denotes the Calder\'on operator (see~\cite[p.~90]{BanSay22}), we can
  apply \cite[Lmm.~4.8]{BanSay22} to obtain the following result. 

  \begin{lemma}
    Let $s\in\C_+$, then there exists $C>0$ such that
    \begin{equation}
      \label{eq:EstCoercivityOpCcal}
      \real \Bigl\langle \Ccal(s) \begin{bmatrix}
        \varphi \\ \psi
      \end{bmatrix} , \begin{bmatrix}
        \varphi \\ \psi
      \end{bmatrix} \Bigr\rangle_{\Xcal',\Xcal}
      \,\geq\, C \min(1,|s|^2) \frac{\real(s)}{|s|^2}
      \bigl(\|\varphi\|_{\HminhalfdiD}^2 + \|\psi\|_{\HhalfdiD}^2\bigr) 
    \end{equation}
    for all $(\varphi,\psi)\in \Xcal$. 
    In particular, the operator $\Ccal(s):\Xcal\to \Xcal'$ is coercive. 
  \end{lemma}

  Accordingly, the operator $\Ccal(s)$ is invertible for all
  $s\in\C_+$, and a bound for the norm of its inverse can be obtained
  similarly to~\cite[Lmm.~7.2]{BanSay22}. 

  \begin{lemma}
    There exists $C_\sigma>0$ such that
    \begin{equation}
      \label{eq:EstOpnormOpCcal}
      \|\Ccal^{-1}(s)\|_{\Xcal\leftarrow \Xcal'} 
      \,\leq\, C_\sigma \frac{|s|^2}{\real(s)} \,,
      \qquad \real(s)\geq\sigma>0 \,.
    \end{equation}
  \end{lemma}

  \begin{proof}
    Let $(g, h)\in \Xcal'$. 
    Since $\Ccal(S)$ is invertible, we can define
    \begin{equation*}
      \begin{bmatrix}
        \varphi \\ \psi
      \end{bmatrix}
      \,:=\, \Ccal^{-1}(s) \begin{bmatrix}
        g \\ h
      \end{bmatrix} \,.
    \end{equation*}
    Using \eqref{eq:EstCoercivityOpCcal}, we obtain
    \begin{equation*}
      \begin{split}
        C &\min(1,|s|^2) \frac{\real(s)}{|s|^2}
        \Bigl(\|\varphi\|_{\HminhalfdiD}^2 + \|\psi\|_{\HhalfdiD}^2\Bigr) 
        \,\leq\, \biggl| \Bigl\langle \Ccal(s) \begin{bmatrix}
          \varphi \\ \psi
        \end{bmatrix} , \begin{bmatrix}
          \varphi \\ \psi
        \end{bmatrix} \Bigr\rangle_{\Xcal',\Xcal} \biggr| \\
        &\,=\, \biggl| \Bigl\langle \begin{bmatrix}
          g \\ h
        \end{bmatrix} , 
        \begin{bmatrix}
          \varphi \\ \psi
        \end{bmatrix} \Bigr\rangle_{\Xcal',\Xcal} \biggr|
        \,\leq\, \Bigl(\|h\|_{\HminhalfdiD}^2
        + \|g\|_{\HhalfdiD}^2\Bigr)^{\frac12}  
        \Bigl(\|\varphi\|_{\HminhalfdiD}^2
        + \|\psi\|_{\HhalfdiD}^2\Bigr)^{\frac12} \,.
      \end{split}
    \end{equation*}
    This shows \eqref{eq:EstOpnormOpCcal}. 
  \end{proof}

\end{appendix}

%%%%%%%%%%%%%%%%%%%%%%%%%%%%%%%%%%%%%%%%%%%%%%%%%%%%%%%%%%%%%%%%%%%%%% 
\section*{Acknowledgments}
%%%%%%%%%%%%%%%%%%%%%%%%%%%%%%%%%%%%%%%%%%%%%%%%%%%%%%%%%%%%%%%%%%%%%% 
This research was funded by the Deutsche Forschungsgemeinschaft (DFG, 
German Research Foundation) -- Pro\-ject-ID 258734477 -- SFB 1173.
The work of M.K.\ was also supported by the Research Council of Finland 
(Flagship of Advanced Mathematics for Sensing, Imaging and Modelling 
grant~359182).

%%%%%%%%%%%%%%%%%%%%%%%%%%%%%%%%%%%%%%%%%%%%%%%%%%%%%%%%%%%%%%%%%%%%%% 
{\small
  \bibliographystyle{abbrvurl}
  \bibliography{temporal_domain_derivative}
}

\end{document}